\documentclass{article}
\usepackage{graphicx}

\title{Fast Hermitian Diagonalization with Nearly Optimal Precision}
\author{Rikhav Shah}
\date{June 2024}

\usepackage{graphicx}
\usepackage{amsmath,amssymb,amsthm}
\usepackage{algpseudocode}
\usepackage{algorithm}
\usepackage{fullpage}
\usepackage{hyperref}
\usepackage{xcolor}
\usepackage{mleftright}
\usepackage{enumitem}

\newcommand{\C}{{\mathbb{C}}}

\newcommand{\eps}{\varepsilon}
\newcommand{\normal}{{\mathcal N}}

\renewcommand{\Re}{\mathrm{Re}}
\renewcommand{\Im}{\mathrm{Im}}

\newcommand{\wt}{\widetilde}

\newcommand{\abs}[1]{\mleft|#1\mright|}
\newcommand{\magn}[1]{\left\|#1\right\|}
\newcommand{\rmagn}[1]{\|#1\|}

\newcommand{\pare}[1]{\mleft(#1\mright)}

\newcommand{\set}[1]{{\left\{{#1}\right\}}}
\newcommand{\sqbrac}[1]{{\left[{#1}\right]}}
\newcommand{\bmat}[1]{\begin{bmatrix}#1\end{bmatrix}}

\newcommand{\alg}[1]{\textnormal{\texttt{#1}}}

\newcommand{\ceil}[1]{\mleft\lceil#1\mright\rceil}

\newcommand{\spliteq}[2]{\begin{equation}#1\begin{split}#2\end{split}\end{equation}}

\DeclareMathOperator*{\poly}{poly}

\DeclareMathOperator*{\rank}{rank}
\DeclareMathOperator*{\tr}{tr}

\DeclareMathOperator*{\sign}{sign}
\DeclareMathOperator*{\range}{range}

\newtheorem{theorem}{Theorem}[section]
\newtheorem{lemma}[theorem]{Lemma}
\newtheorem{proposition}[theorem]{Proposition}

\theoremstyle{definition}
\newtheorem{definition}{Definition}[section]

\newtheorem{remark}{Remark}

\algdef{SE}[REPEATN]{RepeatN}{End}[1]{\algorithmicrepeat\ #1 \textbf{times}}{\algorithmicend}
\newcommand{\meps}{\mathbf{u}}

\newcommand{\mmu}{\mu_{\alg{MM}}(n)}
\newcommand{\qrmu}{\mu_{\alg{QR}}(n)}
\newcommand{\nmu}{\sqrt nc_{\alg N}}

\newcommand{\udeflate}{\meps_{\alg{DEFLATE}}}
\newcommand{\usign}{\meps_{\alg{SIGN}}}

\newcommand{\flops}{floating point operations}
\newcommand{\flop}{floating point operation}

\newcommand{\thealg}{\alg{EIGH-INTERNAL}}

\begin{document}

\maketitle

\begin{abstract}
    Algorithms for numerical tasks in finite precision simultaneously seek to minimize the number of floating point operations performed, and also the number of bits of precision required by each floating point operation.
    This paper presents an algorithm for Hermitian diagonalization requiring only $\lg(1/\eps)+O(\log(n)+\log\log(1/\eps))$ bits of precision where $n$ is the size of the input matrix and $\eps$ is the target error.
    Furthermore, it runs in near matrix multiplication time.
    
    In the general setting, the first complete analysis of the stability of a near matrix multiplication time algorithm for diagonalization is that of Banks et al \cite{b0}. They exhibit an algorithm for diagonalizing an arbitrary matrix up to $\eps$ backward error using only $O(\log^4(n/\eps)\log(n))$ bits of precision.
    This work focuses on the Hermitian setting, where we determine a dramatically improved bound on the number of bits needed.
    In particular, the result is close to providing a practical bound.
    The exact bit count depends on the specific implementation of matrix multiplication and QR decomposition one wishes to use, but if one uses suitable $O(n^3)$-time implementations, then for $\eps=10^{-15},n=4000$, we show $92$ bits of precision suffice (and 59 are necessary). By comparison, the same parameters in \cite{b0} does not even show that $682,916,525,000$ bits suffice.

\end{abstract}

\section{Introduction}
\label{a0}
This paper considers Hermitian diagonalization in finite arithmetic. Given a Hermitian matrix $A$ and target accuracy $\eps$, our goal is to compute nearly unitary $U$ and exactly diagonal $D$ such that $\magn{A-UDU^*}\le\eps\magn A$ with high probability.
An algorithm for such a task can be evaluated on two primary metrics: the number of \flops\,performed (which we call the runtime) and the number of bits of precision required for each \flop. 

Despite the widespread, highly successful use of procedures for this task in practice, the literature long lacked precise guarantees of their performance in finite arithmetic. It's common to treat $n$ and $\eps$ as formal variables satisfying the relations $\poly(n)\cdot\eps=\eps$ and $\eps^2=0$. This dramatically reduces the complexity of proofs as one only needs to keep track of the existence of a single error term, but introduces the possibility of subtle mistakes; namely one must be careful that the coefficients appearing in the $\poly(n)$ factor are absolute constants, and one also cannot apply either of those relations more than $\poly(n)$ times in the proof. This approach also precludes the determination of precise, quantitative bounds on stability. 
In general, the bound numerical analysts typically aim for is showing that $\lg(1/\eps)+O(\log n)$ bits suffice, or stated another way, that backward error is $\poly(n)\meps$ where $\meps^{-\#\text{ bits of precision}}$ is machine precision.
For Hermitian diagonalization in particular, Proposition \ref{a1} shows that no fewer than $\lg(1/\eps)+0.5\lg(n)-2$ bits suffice.

Most algorithms used and analyzed fall into two categories: QL/QR-algorithms first introduced by Francis \cite{b1,b2} and spectral divide-and-conquer algorithms first introduced by Beavers and Denman \cite{b3,b4,b5}.
These algorithms, as is thematic in this field, are frequently built on top of three essential primitives: matrix multiplication, QR decomposition, and matrix inversion. And so, the stability and runtime of algorithms depend on the stability and runtime of the implementations of these primitives. There is a trade-off among the best algorithms between runtime and stability. On the slow side, implementations of each of these primitives exist using $O(n^3)$ time requiring only $\lg(1/\eps)+\lg(n)+O(1)$ bits to achieve $\eps$ backward error. Faster implementations use the innovations of \cite{b6,b7} which show for each $\eta>0$ an implementation of matrix multiplication and QR decomposition using $O(n^{\omega+\eta})$ time requiring only $\lg(1/\eps)+O(\log(n))$ bits, where $O(n^\omega)$ is the speed of matrix multiplication in exact arithmetic. It also shows an implementation of inversion with the same runtime using $\lg(1/\eps)+O(\log(n)\log(\kappa(A)))$ bits to achieve a \textit{forward} error of $\eps$.

Unfortunately, all published QL/QR-algorithms use $O(n^3)$ time as they require reduction of a matrix to tridiagonal or Hessenberg form.
In exact arithmetic, the first globally convergent QR-algorithm for Hermitian matrices was proposed by \cite{b8}. Both \cite{b9,b10} 
bound the rate of convergence of that algorithm, resulting in an $O(n^3+n^2\log(1/\eps))$ time algorithm. \cite{b11} proposed a more expensive version that globally converges in the non-Hermitian setting as well. The follow-up work \cite{b12} analyzed it in finite precision, finding that it uses $\wt O(n^3)$ operations performed with $O(\log(n/\eps)^2\log\log(n/\eps))$ bits of precision plus $\wt O(n)$ operations performed with $O(\log(n/\eps)^4\log\log(n/\eps)^2)$ bits of precision. They also conjecture that $O(\log(n/\eps))$ bits suffice in the Hermitian setting.

Spectral divide-and-conquer algorithms better exploit fast primitives.
Many works can be adapted to form part of the ``divide'' step. In particular, computing the matrix sign or polar decomposition can both be used in the Hermitian case.
A history up to 1995 of algorithms computing the matrix sign can be found in Section 1 of \cite{b13}.  Since then, multiple works have appeared, typically assuming the input matrix is reasonably well conditioned\footnote{This assumption is satisfied with high probability by adding a random shift to the matrix.}. We highlight three algorithms for matrix sign / polar decomposition that have been explicitly incorporated into diagonalization procedures.%

Newton iteration: this method involves matrix inversion, and so the overall stability guarantee is substantially worse if one wants a near matrix multiplication time algorithm. Assuming access to stable inversion,
\cite{b14} shows a scaled version of Newton Iteration converges stably. \cite{b15} gives a much shorter argument to the same effect, but \cite{b16} points out a flaw in the proof arising from the ``$\poly(n)\cdot\eps=\eps$'' framework. \cite{b17} provides a much more compact proof in a way that generalizes to several other algorithms.
Finally, \cite{b0} provides a complete, unconditional analysis of Newton iteration using fast inversion, and furthermore incorporates that method into a full end-to-end analysis of a divide-and-conquer algorithm. They show that their diagonalization algorithm runs in near matrix multiplication time and succeeds when using $O(\log^4(n/\eps)\log(n))$ bits of precision. This was the first concrete bound appearing for any algorithm for diagonalization, and remains the best known bound among algorithms running in near-matrix multiplication time.

QR-based dynamically weighted Halley (QDWH): \cite{b18} finds a quickly-converging iterative scheme, QDWH, for computing the matrix sign built on top of QR decomposition. Unfortunately, the proof of stability of QDWH appearing in \cite{b17} (indeed the only known proof) requires a stronger notion of stability of QR decomposition than what \cite{b7} shows can be achieved quickly\footnote{\cite{b17} requires a ``per-row'' backward error guarantee. That is, $A,A'$ appearing in the leftmost inequality of (\ref{a2}) must be replaced by $e_j^*A,e_j^*A'$ for every $j\in[n]$. The difference in these guarantees is the most stark when the lengths of the rows of $A$ are quite different.}.
As a consequence, we only know how to stably perform QDWH in $O(n^3)$ time. \cite{b19} shows how to build a Hermitian eigen-decomposition on top of QDWH, following the usual divide-and-conquer setup.
However, the proof of its stability is incomplete, with at least a couple limitations.
The first limitation is stated explicitly as condition number 2. A crucial step of divide-and-conquer is computing a basis for the range of a projection matrix. \cite{b19} assumes this can be done stably; this paper finally provides that required analysis in Section \ref{a3}. The second limitation is choice of the ``split points''. \cite{b19} recommends a couple different techniques for picking the split points, which determine the size of the sub-problems. But the worst case guarantee is that a problem of size $n$ gets split into problems of size $1$ and $n-1$. This means the overall method may take $O(n^4)$ time.

Implicit repeated squaring (IRS): \cite{b20} improves upon the IRS method of \cite{b21} using \cite{b7} thereby giving a stable algorithm for computing an analog\footnote{Let $P_{(*)}$ denote the projection onto the span of the eigenvectors whose eigenvalues satisfy $(*)$. Then the matrix sign is $P_{(\cdot)>0}-P_{(\cdot)<0}$. \cite{b20,b21} compute $P_{\abs\cdot>1}-P_{\abs\cdot<1}$} of the matrix sign in near matrix multiplication time.
\cite{b20} also provides finite-arithmetic analysis of the full ``divide'' step; namely it tightens the analysis of rank-revealing URV appearing in \cite{b7} to show how to convert the approximate spectral projectors into approximate bases of the invariant subspaces.
However, they do not explicitly bound the error of the computed bases nor the accumulation of error throughout the entire divide-and-conquer algorithm.

The work \cite{b24} extends a key subroutine of \cite{b0} to the Hermitian pencil case, but does not reduce the precision required.
Another algorithm of note is that of \cite{b25}. They consider the general eigenpair problem and find an algorithm using homotopy continuation and running in $O(n^{10}/\eps^2)$ in general \footnote{The result is stated as $n^9/\sigma^2$ where $\sigma$ is the standard deviation of an normalized Gaussian perturbation, which must be order $\eps/\sqrt n$ for the desired error.}. They provide only an informal argument that their method is stable. 

\subsection{Contributions}
The main result of this paper is Theorem \ref{a4}. It presents an algorithm for Hermitian diagonalization running in near matrix multiplication time and shows it requires only $\lg(1/\eps)+O(\log(n)+\log\log(1/\eps))$ bits of precision\----near the optimal dependence on $\eps$ and linear in the optimal dependence on $n$.
This paper offers several contributions.

Proposition \ref{a1} states a concrete lower bound on required bits of precision; in particular, it shows the existence of two matrices with far apart true diagonalizations that would get rounded to the same matrix without at least $\lg(1/\eps)+0.5\lg(n)-2$ bits of precision.

Section \ref{a5} provides a rigorous quantitative analysis of the stability of Newton-Schulz iteration for computing the matrix sign function in the Hermitian setting. 
This iteration has been considered \cite{b26,b17} but quantitative bounds on it's stability have not appeared.
We show it to be significantly more stable than Newton iteration for matrix sign, which is used by \cite{b0}, though it only succeeds in the Hermitian setting.
Section \ref{a6} discusses this difference in stability in depth\----the improvements are owed both to the orthgonality of eigenvectors in this setting and to the fact Newton-Schulz is inverse-free.
This analysis enables the algorithmic improvement of replacing Newton with Newton-Schulz, resulting in a significant reduction in the precision's dependence on $\eps$\----namely $2\cdot\lg(1/\eps)$ compared to $O(\lg(1/\eps)^4)$ of \cite{b0}.

The next contribution, in Section \ref{a3}, is a streamlined version of the deflation algorithm appearing in \cite{b0}, along with a much tighter analysis. This strengthened analysis reduces the bit dependence on $\eps$ from $2\cdot\lg(1/\eps)$ to $1\cdot\lg(1/\eps)$ and saves several $\lg(n)$ terms, and therefore brings us to near-optimal precision.

Finally, in Section \ref{a7} we highlight two additional differences between our spectral bisection method and that of \cite{b0}.
The first difference\----discussed in Remark \ref{a8}\----is an adaptive setting of parameters used by spectral bisection. The most straightforward approach, used by \cite{b0}, geometrically decreases $\eps$ in recursive calls. But smaller $\eps$ leads to both a longer runtime and a higher bit requirement. We manage to decrease $\eps$ more slowly. The key to enabling this is tracking the depth of recursion and decreasing $\eps$ depending on the current depth. This again significantly reduces our bit requirement.
The second difference\----discussed in Remark \ref{a9}\----is the elimination of spectral shattering, which played a central role in the algorithm of \cite{b0}. Instead of shattering, we implement the suggestions in \cite{b20}: picking shifts randomly (as opposed to using binary search to find a good split point) and adding an additional base case (rather than just the $n=1$ case).

The rest of the introduction is dedicated to background and preliminaties. Section \ref{a5} addresses the primary bottleneck of \cite{b0}, which is the computation of the matrix sign function. Section \ref{a3} provides stronger analysis of deflate which allows us to achieve near-optimal dependence on $\eps$. Section \ref{a7} puts these pieces together into a spectral bisection method for Hermitian diagonalization.

\subsection{Model of computation}
We adopt the standard model of floating point arithmetic.
Numbers which are stored exactly are called floating point numbers.
For each $z\in\C$,
there exists a floating point number $\alg{fl}(z)$ satisfying
\[\abs{\alg{fl}(z)-z}\le\meps\abs z.\]
For each operation $\circ\in\set{+,-,\times}$ and pair of floating point numbers $x,y$, the result of computing $x\circ y$ in floating point arithmetic yields
\begin{equation}
\label{a10}
\abs{\alg{fl}(x\circ y) - x\circ y}\le\meps\abs{x\circ y}.
\end{equation}
Additionally, division by two can be done exactly. That is, $\alg{fl}(x/2)=\alg{fl}(x)/2$.
All nonzero values used in our algorithm are polynomial in $n,\eps^{-1}$, so we can avoid underflow and overflow errors by using $\lg\lg(n/\eps)+O(1)$ bits to store the exponent.

\begin{remark}[Double-double precision]
    One may simulate precision $\meps^2$ by approximating a real (or complex) $x$ by the sum $y+z$ where $y=\alg{fl}(x)$ and $z=\alg{fl}(x-y)$. This suggests a natural trade-off between precision and runtime: by increasing the number of \flops\, by a constant factor, one can get away with a constant factor reduction in the number of bits used. Though theoretically sound, this technique has serious drawbacks in practice. Namely, hardware implementations of various matrix operations assume that each entry of an input fits into a single machine word. That is, each entry of an  input matrix is a single floating point number and isn't abstractly expressed as the sum of two values. Because of this, we build our algorithm on top of subroutines that don't allow inputs to be formatted like this. Throughout this paper, every real or complex number is approximated by a single floating point number. With this constraint, it is possible to obtain a lower bound on the precision required; see Proposition \ref{a1}.
\end{remark}

\begin{proposition}[Lower bound]
\label{a1}
Any method that computes $U,D$ for a given symmetric $A$ satisfying $\magn{A-UDU^*}\le\eps\magn A$
requires $\lg(1/\meps)\ge\lg(1/\eps)+0.5\lg(n)-2$ bits of precision.
\end{proposition}
\begin{proof}
Let $A$ be an $n\times n$ Hadamard matrix, i.e. $\abs{a_{ij}}=1$ and $\magn A=\sqrt n$. Let $(U,D)$ be the result of diagonalizing $A$.
Consider the number of positive versus negative entries of the residual $A-UDU^*$.
If there are more positive entries, set $B$ to be the all ones matrix, otherwise set it to be the all minus one matrix.
Note $A':=A+(\meps/2)B$ will be stored exactly as $A$ in floating point arithmetic, so the result of diagonalizing $A'$ will again be $(U,D)$.
Thus the residual of diagonizing $A'$ is $A'-UDU^*=A-UDU^*+(\meps/2)B$. But by construction of $B$, at least half the entries of this residual have absolute value at least $\meps/2$. In particular, it's norm is at least ${\meps n}/4$, which we desire to be at most $\sqrt n\eps$. So we need $\meps\le4\eps/\sqrt n$.
\end{proof}

\newcommand{\unif}{\alg{Unif}}
\renewcommand{\normal}{\alg{Normal}}

\subsection{Subroutines}
The algorithm of this paper is built on top of several essential primitives. We assume black-box access to the following four methods.
\begin{definition}[From \cite{b6}]
\label{a12}
$\alg{MM}$ is an algorithm for matrix multiplication using $T_{\alg{MM}}(n)$ \flops\,satisfying
\begin{equation}\label{a13}
    \magn{\alg{MM}(A,\,B)-AB}\le\mu_{\alg{MM}}(n)\meps\cdot\magn A\cdot\magn B.
\end{equation}
$\mu_{\alg{MM}}(n)$ is a low-degree polynomial in $n$.
If $A=B^T$, then $\alg{MM}(A,B)$ will be exactly Hermitian.
For simplicity of many bounds we assume $\mmu\ge10$. We also assume $T_{\alg{MM}}(n)$ is convex.
\end{definition}
\begin{definition}[From \cite{b7}]
\label{a11}
$[Q,R]=\alg{QR}(A)$ is an algorithm for matrix multiplication using $T_{\alg{QR}}(n)$ \flops\,satisfying for some matrix $A'$ and unitary matrix $Q'$,
\begin{equation}\label{a2}
(Q')^*A'=R\text{ is upper triangular}
\quad\And\quad
\magn{Q-Q'}\le\mu_{\alg{QR}}(n)\meps
\quad\And\quad
\magn{A-A'}\le\mu_{\alg{QR}}(n)\meps\cdot\magn A.
\end{equation}
$\mu_{\alg{QR}}(n)$ is a low-degree polynomial in $n$. We assume $T_{\alg{QR}}(n)$ is convex.
\end{definition}

\begin{definition}
\label{a15}
$\unif$ is a method for computing approximately uniform random samples from symmetric intervals. It should use only a constant $T_\unif$ number of operations and satisfy the following.
If $c'$ is a random variable distributed uniformly on the real interval $[-s,s]$,
there exists a coupling of $\unif(s)$ and $c'$ such that
\(\unif(s)\in[-s,s]\And\abs{\unif(s)-c'}\le s\meps\)
with probability 1.
\end{definition}
\begin{definition}
\label{a16}
$\normal$ is a method for computing approximately normal random samples. It should use only a constant $T_{\alg N}$ number of operations.
Let $z$ be a complex Gaussian where $\Re(z)$ and $\Im(z)$ are i.i.d. Gaussian samples with mean 0 and variance $1/2$. Then for a constant $c_{\alg N}$, there exists a coupling of $\normal()$ and $z$ such that
\(\abs{\normal()-z}\le \abs{z}c_{\alg N}\meps\)
with probability 1.
\end{definition}

\section{Matrix sign function}
\label{a5}
\subsection{Insufficiency of Newton iteration}
\label{a6}
The bottleneck in the algorithm of \cite{b0} is in the estimation of the matrix sign function.
The iterative scheme used is the same one initially proposed by \cite{b3}:
\begin{equation}
\label{a19}
    A_0=A\quad\And\quad A_{k+1}=\frac{A_k+A_k^{-1}}2.
\end{equation}
In exact arithmetic, this is equivalent to running Newton iteration for the system $\lambda^2-1=0$, which enjoys quadratic convergence to $\pm1$ everywhere except for $\Re(\lambda)=0$. So for matrices with no eigenvalues on the imaginary axis, we have $A_k\to\sign(A)$ as $k\to\infty$.
In finite arithmetic, a crucial step in the proof of convergence is showing that the pseudospectrum of $A_k$ does not grow too quickly with $k$. This becomes more and more difficult the closer the eigenvalues of $A_k$ get to each other, so represents an obstacle for any iterative scheme, not just (\ref{a19}). In the bounds of \cite{b0}, the bit complexity required to handle this growth depends \textit{exponentially} in the number of iterations used.
They show at most
\[
\lg(1/(1-\alpha_0))
+3\lg\lg(1/(1-\alpha_0))
+\lg\lg((1/\beta\eps_0))
+O(1)
\]
iterations are needed, where for the purposes of this discussion you can think of $1/(1-\alpha),\beta,\eps_0$ as all being polynomials in $n,\eps^{-1}$. Each $1\cdot\lg\lg(x)$ term contributes another \textit{factor} of $\lg(x)$ to the bit complexity. 
But since these are not the dominate terms in the expression, shaving off the logs from the bit complexity of this algorithm, if possible, requires much care.
The Hermitian setting completely sidesteps this issue since the pseudospectra of Hermitian matrices are always well-behaved.

The secondary source of instability relates to this particular iterative scheme.
Specifically, to the computation of $A_k^{-1}$.
If one had a ``backward-stable'' algorithm $\alg{INV}$ for computing inverses such that for every $A$ there existed a small $E$ satisfying $\alg{INV}(A)=(A+\magn{A}E)^{-1}$, some algebra reveals that one would also obtain the ``forward error'' bound
\[\magn{A^{-1}-\alg{INV}(A)}\le O(\magn{E})\cdot\kappa(A).\]
Unfortunately, no such algorithm running in near matrix multiplication time is known. The closest we have is the work of \cite{b7}, which finds a \textit{logarithmically stable} algorithm, i.e. one satisfying the weaker guarantee that
\[\magn{A^{-1}-\alg{INV}(A)}\le O\pare{\magn{E}}\cdot\kappa(A)^{O(\log n)}.\]
So \cite{b0} must perform everything with enough precision to promise that $\magn{E}\ll\kappa(A)^{-O(\log n)}$. This necessitates another factor of $O(\log\kappa(A)\cdot\log(n))$ bits of precision.
For this reason, an `inverse-free' method that does not use inversion (even well-conditioned inversion) is desirable. %
This motivates the iterative scheme
\begin{equation}
\label{a20}
    A_0=A\quad\And\quad A_{k+1}=\frac{3A_k-A_k^3}2=\frac12\cdot A_k\cdot{(3I-A_k^2)}
    \end{equation}
corresponding to Newton-Schultz iteration, 
which in this instance is equivalent to Newton iteration for the system $\lambda^{-2}-1=0$.
This system has been considered \cite{b26,b17} but quantitative bounds on it's stability have not appeared.
Each iteration uses only two calls to a method for multiplying matrices and no other expensive primitives.
This method is insufficient for the non-Hermitian problem as it does not converge for many complex starting points $z$; for instance it does not converge when $\abs{\Im(z)}\ge2\abs{\Re(z)}$. Figure \ref{a21} shows the regions of convergence for the two methods.
Fortunately, it converges on the interval $(-\sqrt5,\sqrt5)$ in the real line, which is enough for Hermitian diagonalization.

\begin{figure}[h]
\begin{center}
\includegraphics[scale=0.5]{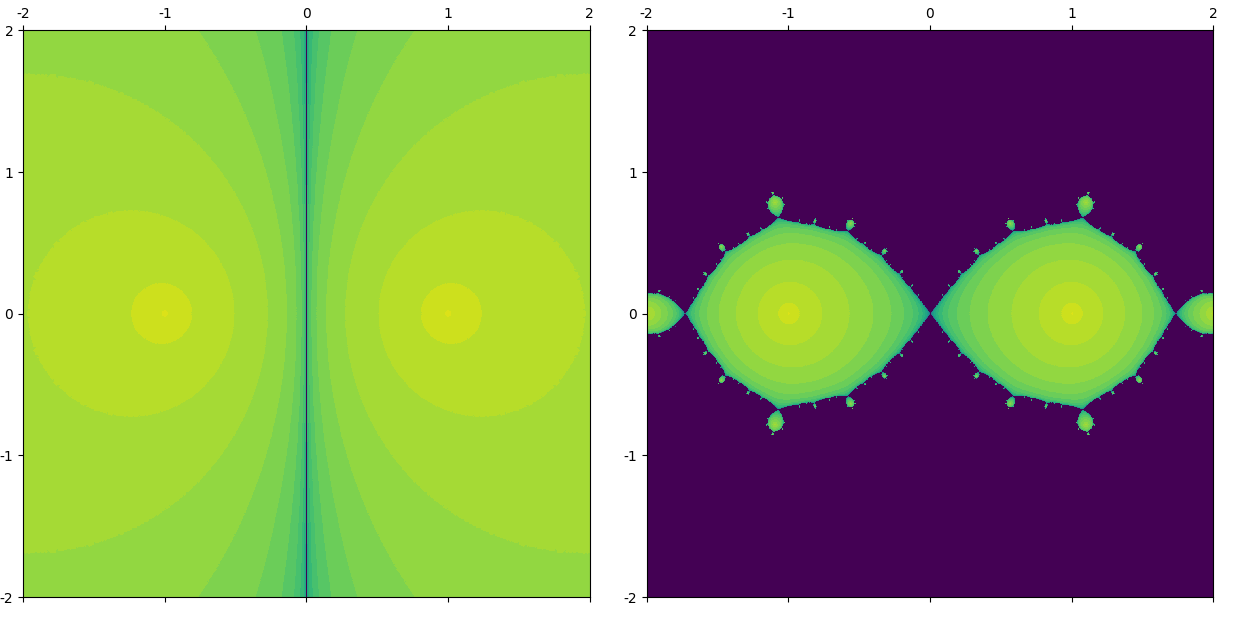}
\end{center}
\label{a21}
\caption{
Convergence plots for Newton iteration (left) and Newton-Schulz iteration (right).
The color denotes the number of iterations $k$ until $\abs{z_k^2-1}<10^{-15}$.
The lightest shade of yellow is one iteration and each ring denotes one additional iteration. Purple denotes ``does not converge''.}
\end{figure}

\begin{remark}[Polar decomposition]
For Hermitian matrices, computing the polar decomposition recovers the matrix sign since
$A=\sign(A)(A^*A)^{1/2}$ is the unique factorization of $A$ into the product of a unitary matrix $\sign(A)$ and a positive definite matrix $(A^*A)^{1/2}$.
\end{remark}

\begin{remark}[Faster iterative algorithms]
As we will see later, the number of iterations required for this method to converge depends on the condition number of the input matrix.
This paper is mostly concerned with the precision required, which in the Hermitian setting depends only mildly on the number of iterations required.
Additionally, random shifts are employed to ensure the condition number is never too large.
So analysis of faster iterative schemes is left to future work. Nevertheless it's worth discussing some candidates. QDWH of Nakatsukasa and Higham \cite{b18} for polar decomposition was shown to be stable in \cite{b17}, and experimentally was found to converge in just 6 iterations for $\eps=10^{-15}$ regardless of $n$ and $\kappa$.
However, the proof of stability requires QR-decomposition with a \textit{per-row} backward error. That is, each row of the input matrix may receive a multiplicative norm-wise perturbation, irrespective of the norms of the other rows. This can be accomplished via Householder reflections with pivoting in $O(n^3)$ time, but no reduction to matrix multiplication is known.
Their scheme is a `weighted' version of $A_{k+1}=(3A_k+A_k^3)(1+3A_k^2)^{-1}$
which adaptively changes the coefficients of the iterates as the method is run.
An even higher degree rational function is used by the method Zolo-pd of Nakatsukasa and Freund, which experimentally converges in just 2 iterations \cite{b29}, though no proof of stability is known.
Several more morally similar variants of that and (\ref{a19}), (\ref{a20}) can be found in \cite{b13,b30,b18,b31}.
These methods frequently require some additional knowledge about the matrix (e.g. QDWH) requires a crude upper estimate on the condition number) which can typically be computed in matrix multiplication time.
These methods present good candidates for replacing (\ref{a20}), but require complete analysis of stability and runtime in finite precision.
\end{remark}

As motivated in the previous subsection, we use Newton-Schulz iteration to approximate the matrix sign function.
Fix
\[g(x)=\frac{3x-x^3}2=\frac12\cdot x\cdot(3-x^2)\]
and implement it using $\alg{MM}$ by
\begin{equation}
\label{a22}
    \alg{g}(A)=\frac12\alg{MM}\pare{A,\pare{3I-\alg{MM}(A,A)}}.
\end{equation}
In exact arithmetic, by the functional calculus applying (\ref{a20}) to a Hermitian matrix is equivalent to applying (\ref{a20}) to it's eigenvalues. We need to argue that this equivalence does not break too much in the presence of round-off error. We also need to analyze the convergence of (\ref{a20}) in the presence of error when applied to scalars.
The next two lemmas accomplish the first task.
We start by bounding the forward error of computing a single iteration.
\begin{lemma}[One step error bound]
\label{a23}
Assume $\meps\le\min(1/3,\mmu^{-1})$.
Then $\alg{g}$ specified in (\ref{a22}) satisfies
\[\magn{\alg{g}(A)-g(A)}\le\mu_{\alg g}(n,\magn A)\meps\]
for
\[\mu_{\alg g}(n,a):=\frac12\pare{7+(6+\mmu)a^2}a.\]
Furthermore, such an implementation requires only
\[T_{\alg g}(n)=2T_{\alg{MM}}(n)+n^2+n\]
\flops.
\end{lemma}
\begin{proof}
We can numerically compute $g$ as
    In this proof, $E_k$ denote matrices with $\magn{E_k}\le\meps$.
    Let $\alg{MM}(A,A)=A^2+\magn{A}^2\mmu E_1$.
    Let $B$ be the result of numerically subtracting $\alg{MM}(A,A)$ from $3I$. This subtraction incurs an entry-wise multiplicative $(1+\meps)$ error along the diagonal only. In particular, because the error is diagonal the entry-wise absolute value bound is upgraded to a matrix norm bound for free. In particular,
    \begin{align*}
        B&=3I-\alg{MM}(A,A)+\magn{3I-\alg{MM}(A,A)}E_2
       \\&=3I-A^2-\magn{A}^2\mmu E_1+\magn{\pare{3I-A^2-\magn{A}^2\mmu E_1}}E_2
    \end{align*}
    so the forward error of computing $B$ is at most
    \[
    \magn{A}^2\mmu\meps+\pare{3+\magn{A}^2+\magn{A}^2\mmu\meps}\meps
\le
    \pare{3+(2+\mmu)\magn{A}^2}\meps
    \]
    Let $\mu_B=3+(2+\mmu)\magn{A}^2$ so $B=3I-A^2+\mu_BE_3$. Then
    \begin{align*}
    \alg{MM}(A,B)-2g(A)
      &=AB-2g(A)+\magn{A}\magn{B}E_4
    \\&=\mu_B AE_3+\magn{A}\magn{B}E_4.
    \end{align*}
    Finally, division by $2$ can be done exactly by decrementing the exponent.
    So the forward error of computing $g(A)$ is at most
    \begin{align*}
    \frac12\pare{\mu_B\magn{A}\meps+\magn{A}\pare{3+\magn{A}^2+\mu_B\meps}\meps}
      &=\frac12\pare{\mu_B+3+\magn{A}^2+\mu_B\meps}\magn A\meps
    \\&=\frac12\pare{7+(3+\mmu)\magn{A}^2+(2+\mmu)\magn{A}^2\meps}\magn A\meps
    \\&=\frac12\pare{7+(6+\mmu)\magn{A}^2}\magn A\meps
    \end{align*}
    as required.

\end{proof}
\begin{remark}
The work \cite{b17} considers an alternate implementation of $\alg{g}$, namely 
\[\alg{g}(A)=\frac{3A-\alg{MM}(A,\alg{MM}(A,A))}2.\]  
It shows that this implementation is stable, but only gives a qualitative analysis. In particular, as far as this author can tell, this implementation suffers from an extra factor of $n$ in the error bound. This comes since the addition is dense and therefore incurs $\meps$ error in every entry, as opposed to merely the diagonal as in the implementation (\ref{a22}). 
\end{remark}
The next lemma shows adding noise does not change the value of $\sign$ by too much.
\begin{lemma}
\label{a24}
    Let $A,B,\eps$ be such that $0$ is not in the interior of $\Lambda_\eps(A)$ and $\magn{A-B}<\eps$. Then
    \[\magn{\sign(A)-\sign(B)}\le n\cdot\frac{\magn{A-B}}{\eps-\magn{A-B}}.\]
\end{lemma}
\begin{proof}
    Let $\gamma$ be the boundary of a connected component of $\Lambda_\eps(A)$. Say it encloses $k$ eigenvalues.
    Then
\begin{align*}
    \magn{\frac1{2\pi}\oint_\gamma\sqbrac{
    \pare{z-A}^{-1}-\pare{z-B}^{-1}
    }}
&=
    \magn{\frac1{2\pi}\oint_\gamma\sqbrac{
    \pare{z-B}^{-1}(A-B)\pare{z-A}^{-1}
    }}
\\&\le
    \frac1{2\pi}\cdot\text{length}(\gamma)\cdot\magn{A-B}\cdot\frac1{\eps-\magn{A-B}}\cdot\frac1\eps
\\&\le k\cdot\magn{A-B}\cdot\frac1{\eps-\magn{A-B}}.
\end{align*}
By the assumption, none of these components cross the imaginary axis. So we may sum over the appropriate components to form the spectral projector onto all the positive or negative eigenvalues. This gives the desired bound.
\end{proof}

We now tackle converge for scalars.
One notices that $x=-1,0,1$ are fixed points of $g$ and that $g'(-1)=g'(1)=0$, which is needed for quadratic convergence. Unfortunately, the iterations do not converge to $\pm1$ for $x=0$ (where it stays at $0$) and for $x\ge\sqrt 5$ (where it does not converge at all). Additionally, which fixed point it converges to for $x\in\pm[\sqrt3,\sqrt5)$ is difficult to control.
We start by showing monotone convergence, even in the presence of error.
\begin{lemma}[Monotone convergence]\label{a25}
    Fix $u\in(0,3/16]$ and $\abs\xi\le u$. Then
    \begin{align*}
        \sign(x)&=\sign(g(x)+\xi) && \forall x\in\pm\pare{u,\,\sqrt{3}-(\sqrt{3}-1)u}, \\
        \abs{x} &\le\abs{g(x)+\xi}\le1+u && \forall x\in[(8/3)u,\,1-(8/3)u].
    \end{align*}
\end{lemma}
\begin{proof}
We prove the statements for $x\ge0$ and the results for $x\le0$ follow symmetrically by noting that $g$ is odd. $g(x)$ is concave on $[0,\,\sqrt3]$ so is lower bounded by its linear spline with nodes $[0,\,1,\,\sqrt3]$. Namely,
\[
g(x)\ge\begin{cases}
    x&0\le x\le1\\
    \sqrt{3}/(\sqrt{3}-1)-{x}/(\sqrt{3}-1)&1\le x\le\sqrt3
\end{cases}
\]
which is larger than $u$ for $x\in(u,\,\sqrt3-(\sqrt3-1)u)$. This implies $\sign(g(x)+\xi)=1$ establishing the first claim.
Now note the polynomial $g(x)-x=(x-x^3)/2$ is concave on the region $[0,1]$ and so is lower bounded by its linear spline with nodes $[0,\,1/2,\,1]$. Namely,
\[
g(x)\ge\begin{cases}
    (3/8)x & 0\le x\le1/2 \\
    (3/8)(1-x) & 1/2\le x\le1
\end{cases}
\]
which is at least $u$ for $x\in[(8/3)u,\,1-(8/3)u]$. This implies $g(x)-u\ge x$ establishing the second claim. Finally by the triangle inequality $\abs{g(x)+\xi}\le\abs{g(x)}+\abs{\xi}\le1+u$.
\end{proof}
In exact arithmetic, this iteration enjoys quadratic convergence. In finite precision, we have quadratic convergence for a large range of values.
We will analyze the convergence of the iterations using the potential function $m(x)=\abs{1-x^2}$. 
\begin{lemma}[Quadratic convergence]
\label{a26}
When $20\abs{\xi}\le\abs x\le1-\sqrt{10\abs\xi}$, one has
\[m(x)^2\ge m(g(x)+\xi).\]
When $\abs{x}\le\sqrt2$ and $\abs{\xi}\le1$, one has
\[m(x)^2+4\abs\xi\ge m(g(x)+\xi).\]
\end{lemma}
\begin{proof}
Note $g$ is odd and $m$ is even, and that the condition on $\xi$ is symmetric in $x$. So it suffices to prove the statement for $x\ge0$.
Fix any $s\in[-1,1]$.
Simple algebraic manipulations reveal that
\begin{align*}
    \abs x\le0.5 &\implies
    \abs x\le\sqrt{\frac{180-\sqrt{180^2-4\cdot100\cdot41}}{200}} \\ &\implies
    41-180x^2+100x^4\ge0 \\ &\implies
    (30-10x^2-1)^2\ge20^2\cdot(2-x^2) \\ &\implies
    30-10x^2-1\ge20\sqrt{2-x^2} \\ &\implies
    \frac{3-x^2}2-\sqrt{2-x^2}\ge s/{20} \\ &\implies
    g(x)-\sqrt{2x^2-x^4}\ge s/{20} \\ &\implies
    g(x)-sx/20\ge\sqrt{2x^2-x^4} \\ &\implies
    (g(x)-sx/20)^2-1\ge{2x^2-x^4-1} \\ &\implies
    \abs{(g(x)-sx/20)^2-1}\le\abs{1-x^2}^2 \\ &\implies
    m(g(x)-sx/20)\le m(x)^2.
\end{align*}
By the Descartes rule of signs, the polynomial $25x^{4}+60x^{3}+26x^{2}-32x+1$ has at most two positive roots, and changes sign at the points $0,1/4,1/2$. So the polynomial is positive for $x>1/2$. In the below, assume $x\le\sqrt 2$.
\begin{align*}
    x\ge0.5 &\implies
    \pare{x-1}^2\pare{25x^{4}+60x^{3}+26x^{2}-32x+1}\ge0\\ &\implies
    \pare{15x-5x^3-(x-1)^2}^2\ge100x^2\pare{2-x^2}\\ &\implies
    15x-5x^3-(x-1)^2\ge10x\sqrt{2-x^2}\\ &\implies
    g(x)-\sqrt{2x^2-x^4}\ge{(x-1)^2}/10  \\ &\implies
    (g(x)-s{(x-1)^2}/10)^2\ge{2x^2-x^4}  \\ &\implies
    (g(x)-s{(x-1)^2}/10)^2-1\ge{-1+2x^2-x^4}  \\ &\implies
    m(g(x)-s{(x-1)^2}/10)\le m(x)^2.
\end{align*}
Note
\[\min\pare{x/{20},\,{(x-1)^2}/{10}}=\begin{cases}x/20 & x\le0.5 \\ (x-1)^2/10 & x\ge0.5\end{cases},\]
so whenever $\abs\xi\le\min\pare{x/{20},\,{(x-1)^2}/{10}}$ for $x\in[0,\sqrt2]$
we have $m(g(x)+\xi)\le m(x)^2$.
Solving for $x$ in terms of $\abs\xi$ yields the desired result.
One may take $\xi=0$ for any $x$ so $m(g(x))\le m(x)^2$ for all $\abs x\le\sqrt2$.
The derivative of $m$ is bounded by $4$ on the interval $[-2,2]$, and $g(x)+\xi\in[-1,1]$ for $\abs{x}\le\sqrt2,\abs\xi\le1$. So $m(g(x)+\xi)\le m(g(x))+4\abs{\xi}\le m(x)^2+4\abs{\xi}$ establishing the second claim.
\end{proof}

\begin{lemma}[Overall scalar convergence]
\label{a27}
Fix a precision $u$ and tolerance $\eps$ so that $10u\le\eps\le3/80$. Let $x_0\in\pm[20u,1.5]$ and $x_{k+1}=g(x_k)+\xi_k$ for adversarial $\abs{\xi_k}\le u$. Then for
\[N\ge N_{\alg{SCALAR}}(x_0,\eps):= 2.5+2\lg\min(\abs{x_0},0.5)^{-1}+\lg\lg(1/\eps)\]
one has
\[\abs{1-x_N^2}\le\eps.\]
\end{lemma}
\begin{proof}
Note
\begin{align*}
    m(x_k)\le\eps  \implies
m(x_{k+1})\le m(x_k)^2+4u\le\eps^2+4u\le(\eps+2/5)\eps\le\eps
\end{align*}
so it suffices to argue $m(x_k)\le\eps$ for some $k\le N:=N_{\alg{SCALAR}}(x_0,\eps)$.
We first claim $\abs{x_k}\le1+u$ for all $k>0$.
The range of $g$ on the domain $[-1.5, 1.5]$ is $[-1,1]$,
so $x_k\in[-1.5,1.5]\implies x_{k+1}\in[-1-u,1+u]\subset[-1.5,1.5]$. Then note $x_0\in[-1.5,1.5]$ by assumption so the claim follows by induction.
Thus if $\abs{x_k}\ge1$, we'd have $m(x_k)\le(1+u)^2-1=2u+u^2\le\eps$, so we may assume $\abs{x_k}\le1$ for $k=1,\cdots,N$.

Let $S=[20u,1-\sqrt{\eps}]$. Let $M$ be the lowest positive index such that $\abs{x_M}\not\in S$. If $M\le N-2$, then $\abs{x_M}\in(1-\sqrt{\eps},1]$ and
$m(x_M)\le1-(1-\sqrt{\eps})^2=2\sqrt{\eps}$ so
\begin{align*}
    m(x_{M+1})
&\le m(x_M)^2+4u\le(2\sqrt \eps)^2+4u=4\eps+4u\le4.4\eps,
\\
m(x_{M+2})
&\le m(x_{M+1})^2+4u\le(4.4\eps)^2+4u=4.4^2\eps^2+0.4\eps\le\eps,
\end{align*}
which would establish the claim. Let us now bound $M$. If $x_1\not\in S$, then $M=1$. Otherwise, $x_1\in S\subset[(8/3)u,1-(8/3)u]$, Lemma \ref{a25} implies $x_k$ is monotonically increasing in $k$ while in $S$ so in fact $x_1,\cdots,x_{M-1}\in S$. Since $1-\sqrt{\eps}\le1-\sqrt{10u}$, Lemma \ref{a26} implies quadratic convergence for $k\le M$,
\[m(x_{k})\le m(x_{k-1})^2\le m(x_{k-2})^{2^2}\le\cdots\le m(x_1)^{2^{k-1}}.\]
Since $x_{M-1}\in S$, we have $m(x_{M-1})\ge 1-(1-\sqrt{\eps})^2=2\sqrt{\eps}-\eps\ge\sqrt{\eps}$. Taking logs gives
\begin{align*}
    \sqrt{\eps}\le(1-x_1^2)^{2^{M-2}}
   &\implies
    \frac12\lg{\eps}\le2^{M-2}\lg(1-x_1^2)\le-2^{M-2}\cdot\frac1{\log2}\cdot x_0^2
\\ &\implies
    \frac{\log 2}{2x_1^2}\lg{(1/\eps)}\ge2^{M-2}
\\ &\implies
    \lg\pare{\frac{\log 2}{2x_1^2}}+\lg\lg{(1/\eps)}\ge{M-2}
\\ &\implies
    \lg\pare{2\log 2}+2\lg(1/\abs{x_1})+\lg\lg{(1/\eps)}\ge M.
\end{align*}
Note $\abs{x_1}\ge\abs{x_0}$ or else $\abs{x_0}\in[1-(8/3)u,1.5]$ implying $\abs{x_1}\ge0.5$. So $\lg(1/\abs{x_1})\le\lg\min(\abs{x_0},0.5)^{-1}$. Then $\lg(2\log2)<0.5$. Combining those together yields $M\le N-2$ as required.
\end{proof}

We're now ready to state and analyze the algorithm for approximating $\sign(A)$. In the below, $\alg g$ is defined by (\ref{a22}).
\begin{algorithm}[H]
    \caption{$\alg{SIGN}(A,\eps,b)$}
    \label{a28}
\begin{algorithmic}[1]
    \Require Nonsingular Hermitian matrix $A$ with $\magn A\le b$, desired accuracy $\eps$.
    \Ensure $\magn{\alg{SIGN}(A,\eps,b)-\sign(A)}\le\eps$.
    \State $A_0\gets A/b$
    \State $k\gets 0$
    \Repeat
    \State$A_{k+1}\gets\alg{g}(A_k)$
    \State$k\gets k+1$
    \Until$\magn{I-\alg{MM}(A_k,A_k)}_{\max}\le\eps/(4n)$
    \State
    \Return $A_k$
\end{algorithmic}
\end{algorithm}

\begin{theorem}[Main guarantee for $\alg{SIGN}$]
\label{a29}
Let \[N:= N_{\alg{SIGN}}(A,\eps,n):= N_{\alg{SCALAR}}\pare{\frac1{\magn{A^{-1}}\cdot b},\,\frac\eps{8n}}.\]
Algorithm \ref{a28} has the advertised properties
when run with
\[\meps\le\meps_{\alg{SIGN}}(\eps,b,\magn{A^{-1}}):=\frac1{\magn{A^{-1}}\cdot b}\cdot\frac1{4\max\pare{N\cdot n\mu_{\alg g}(n,1.1),n^2}}\cdot\eps\]
using
\[3N\cdot T_{\alg{MM}}(n)+O(Nn^2)\]
\flops.
\end{theorem}
\begin{proof}
Let $a=\magn{A^{-1}}^{-1}$ so that the spectrum of $A/b$ is contained in $[-1,-a/b]\cup[a/b,1]$. Then by Lemma \ref{a24},
\[\magn{\sign(A_0)-\sign(A/b)}\le n\cdot\frac{\magn{A_0-A/b}}{a/b-\magn{A_0-A/b}}\le n\cdot\frac{n\meps}{a/b-n\meps}\le\frac{2n^2}{a/b}\meps.\]
Lemma \ref{a23} implies one can express
\[A_{k+1}=g(A_k)+E_k,\quad\quad\magn{E_k}\le\mu_{\alg g}(n,\magn{A_k})\meps.\]
We claim that $\magn{A_k}\le1.1$ for all $k$. It's clearly true $k=0$. Then inductively, \[\magn{A_{k+1}}\le1+\mu_{\alg g}(n,\magn{A_k})\meps\le1+\mu_{\alg g}(n,1.1)\meps\le1.1\] for our selection of $\meps\le0.1\cdot\mu_{\alg g}(n,1.1)^{-1}$. As a consequence, $\magn{E_k}\le\mu_{\alg g}(n,1.1)\meps$. In particular, each eigenvalue of $A_{k+1}$ is contained in the $\mu_{\alg g}(n,1.1)\meps$ pseudospectrum of $g(A_k)$, i.e. can be expressed as $g(\lambda)+\xi$ where $\lambda$ is an eigenvalue of $A_k$ and $\abs\xi\le\mu_{\alg g}(n,1.1)\meps$.
Therefore, Lemma \ref{a27} means for our selection of $N$ that
that each eigenvalue of $A_N$ satisfies $\abs{1-\lambda_j(A_N)^2}\le\eps/(8n)$.
In particular, \[
\magn{I-\alg{MM}(A_N,A_N)}_{\max}
\le\magn{I-\alg{MM}(A_N,A_N)}
\le2\magn{I-A_N^2}
=2\sup_j\abs{1-\lambda_j(A_N)^2}\le\eps/(4n)\]
so the algorithm terminates no later than iteration $N$.
On the other hand, say $M$ is the iteration on which the algorithm terminates. Then
\[
\frac\eps{4n}
\ge\magn{I-\alg{MM}(A_N,A_N)}_{\max}
\ge\frac12\magn{I-A_N^2}_{\max}
\ge\frac1{2n}\magn{I-A_N^2}
\ge\frac1{2n}\sup_j\abs{1-\lambda_j(A_N)^2}
\]
Therefore since $A_M$ and $\sign(A_M)$ are simultaneously unitarily diagonalizable, we have
\[\magn{A_M-\sign(A_M)}
\le\max_j\abs{\lambda_j(A_M) - \sign(\lambda_j(A_M))}
\le\max_j\abs{\lambda_j(A_M)^2 - 1}
\le\eps/2.\]
Note $\sign\circ\,g=\sign$ on the domain $[-\sqrt3,\sqrt3]$ so
\begin{align*}
\magn{\sign(A_{k+1})-\sign(A_{k})}
  &=\magn{\sign(A_{k+1})-\sign(g(A_{k}))}
\\&=\magn{\sign(A_{k+1})-\sign(A_{k+1}-E_k)}
\\&\le n\cdot\frac{\magn{E_k}}{a/b-\magn{E_k}}
\\&\le n\cdot\frac{\mu_{\alg g}(n,1.1)\meps}{a/b-\mu_{\alg g}(n,1.1)\meps}
\\&\le\frac{2n\mu_{\alg g}(n,1.1)}{a/b}\meps
\end{align*}
where the third step is by Lemma \ref{a24} and last is a rearrangement of the requirement on $\meps$.
By the triangle inequality,
\begin{align*}
\magn{A_M-\sign(A)}
  &=  \magn{A_M-\sign(A_0)}+\magn{\sign(A_0)-\sign(A)}
\\&\le\magn{A_M-\sign(A_M)}+\sum_{j=1}^M\magn{\sign(A_{k})-\sign(A_{k-1})}+\magn{\sign(A_0)-\sign(A/b)}.
\\&\le\frac\eps2+\pare{N\cdot\frac{2n\mu_{\alg g}(n,1.1)}{a/b}+\frac{2n^2}{a/b}}\meps
\\&\le\frac\eps2+\frac\eps2\\&=\eps
\end{align*}
as required.
The number of \flops\,used is $N\cdot(T_{\alg g}(n)+T_{\alg{MM}}(n)+n^2)+n^2$. Using $T_{\alg g}(n)\le2T_{\alg{MM}}(n)+O(n^2)$ from Lemma \ref{a23} gives the final result.
\end{proof}

\begin{remark}[$b$ parameter]
    If one is not supplied with the value $b$ in the call $\alg{SIGN}(A,\eps,b)$, an acceptable value can easily be computed in $O(n^2)$ or $O(T_{\alg{MM}}(n))$ time by taking $b=\magn{A}_F$ or $b=\tr(A^{2p})^{1/(2p)}$. In that case, $\magn{A^{-1}}\cdot b$ would be somewhere in between the condition number and weighted condition number of $A$.
\end{remark}

\section{Analysis of deflate}
\label{a3}
The secondary bottleneck in \cite{b0} and a main limitation of \cite{b19} is the computation of $\alg{DEFLATE}$.
Deflation is a procedure for recovering from a low rank matrix $P$ (often a projection), an orthonormal basis for its range.
To the author's knowledge, up to minor variations, the algorithm we reproduce here is the only one running in near matrix multiplication time\footnote{One candidate alternative is QR-decomposition with pivoting, which is recommended by \cite{b21}, but no efficient reduction to matrix multiplication is known.}.
\cite{b0} conceives of this algorithm as a slight modification of the rank-revealing QR-decomposition analyzed by \cite{b7}.
It turns out to be equivalent to the proposal by \cite{b19}, which conceives of the algorithm as running a single iteration of subspace iteration. However, \cite{b19} glosses over the requirements of the starting matrix, which is a non-trivial part of the analyses of \cite{b0,b7}. 
The algorithm is exceedingly simple and is essentially the following: output the QR-decomposition of the first $\rank(P)$ columns of $PG$ where $G$ is the random ``starting'' matrix.
Notice that the method completely fails if $\rank(PG)<\rank(P)$, and struggles when the $\rank(P)$th singular value of $PG$ is small. In order to ensure this value is sufficiently large often enough, one must include an additional factor of $\poly(n)$ in the precision. This factor is unavoidable. However, a much larger issue appearing in \cite{b0} can be avoided. Their argument for the correctness of $\alg{DEFLATE}$ goes through the following argument: if the input is the projection $UU^*$ and output is $Q$, they first show that $U^*Q$ is close to a unitary matrix $W$. Then they convert this into a bound on $\magn{U-QW^*}$, which is needed by the spectral bisection method.
But this conversion introduces a square-root in the error. To see this, consider $Q=\bmat{\cos\sqrt{2\eps}&\sin\sqrt{2\eps}}^*$ and $U=\bmat{1&0}^*$. In this example, $U^*Q\approx1-\eps$ whereas $\magn{U-Q}\approx\sqrt{2\eps}$.
According to this analysis, in order to get the desired accuracy out of $\alg{DEFLATE}$, one must double the number of bits used to overcome the square-root.
We strengthen this analysis by removing the square-root, thereby reducing the bit requirement by a factor of two. The additional insight allowing this is that not only is $U^*Q$ close to unitary, we can show that $(U^\perp)^*Q$ is close to $0$ where $U^\perp$ is a basis for an orthogonal complement of the range of $U$.
This turns out to be sufficient to remove the square-root, giving a tight analysis (up to constants).

This analysis also removes the restriction appearing in \cite{b0} that the input matrix is close to a matrix satisfying $\rank(A^2)=\rank(A)$. This restriction wasn't an issue for our setting since $A$ is an always an orthogonal projection matrix, but the removal of the restriction may be of independent interest.
Our implementation of $\alg{DEFLATE}$ is the following.
\begin{algorithm}[H]
\caption{\alg{DEFLATE}$(\wt A,r)$}
\label{a31}
\begin{algorithmic}[1]
\Require There exists $A\in\C^{n\times n}$ and parameters $t,x,\beta>0$ such that
\[\magn{\wt A-A}\le\beta\le\frac15\cdot\frac{\sigma_r(A)}{\sigma_1(A)+2}\cdot\frac x{(2\sqrt2+t)\sqrt n}\le1\]
and $\rank(A)=r$.
\Ensure For output $\wt U$, there exists semi-unitary $U\in\C^{n\times r}$ such that $\range(U)=\range(A)$ and
\[\magn{\wt U-U}\le6\cdot\frac{\sigma_1(A)+2}{\sigma_r(A)}\cdot\frac{(2\sqrt2+t)\sqrt n}{x}\cdot\beta\]
with probability $1-2e^{-nt^2}-(r/2)x^2$.
\State$\wt G_{ij}=\normal()\quad\forall i,j\in[n]$
\State$M=\alg{MM}(\wt A,\wt G)$
\State$(Q,R)=\alg{QR}(M)$
\State\Return First $r$ columns of $Q$
\end{algorithmic}
\end{algorithm}
A convenient choice of parameters is $x=\sqrt{\rho/r}$ and $t=\sqrt{\log(4/\rho)/n}$ for some $\rho\in(0,1)$. Then for each $0<\eta\le6/5$, if
\[\beta\le\frac{\rho^{1/2}\eta}{\sqrt{nr}}\cdot\frac{\sigma_r(A)}{\sigma_1(A)+2}\cdot\frac1{12\sqrt2+6\sqrt{\log(4/\rho)/n}}\]
then
\[\magn{\wt U-U}\le\eta\]
with probability $1-\rho$.
\begin{remark}
    The only difference between our method \ref{a31} and the one appearing in \cite{b0} is that we take $G$ to be a matrix of Gaussians rather than a Haar unitary matrix. In particular, \cite{b0} replaces $G$ with first output of $\alg{QR}(G)$.
    However, this requires an additional call to $\alg{QR}$ that is ultimately not necessary and in fact introduces some additional error.
\end{remark}

In the following two lemmas, $G$ is an $n\times n$ matrix with i.i.d. complex Gaussian entries.
\begin{lemma}[Theorem 3.2 from \cite{b33}]\label{a32}
    \[\Pr\pare{\sigma_n(G)\le x}\le\frac n2x^2.\]
\end{lemma}
\begin{lemma}[Lemma 2.2 from \cite{b34}]\label{a33}
    \[\Pr\pare{\magn G>(2\sqrt2+t)\sqrt n}\le2e^{-nt^2}.\]
\end{lemma}

\begin{theorem}
\label{a30}
\alg{DEFLATE} has the advertised guarantee when run with precision
\[\meps\le\udeflate(\beta,n)=\frac\beta{4\qrmu+2\nmu+2\mmu}.\]
Furthermore, it uses only
\[T_{\alg{MM}}(n)+T_{\alg{QR}}(n)+T_{\alg N}\cdot n^2.\]
\flops.
\end{theorem}
\begin{proof}
The runtime is clear the algorithm consists of $n^2$ calls to $\normal$, one call to $\alg{MM}$, and one call to $\alg{QR}$ and no additional work.
Throughout the proof we use $(\cdot)_1$ to denote the first $r$ columns of a matrix and use $(\cdot)_{11}$ to denote the upper-left $r\times r$ submatrix. In particular, the output of the algorithm is $Q_1$.
By assumption, we may express $A=U\Sigma V^*$ and $\wt A=A+E$ for $U,V\in\C^{n\times r}$, $\Sigma\in\C^{r\times r}$, and $\magn E\le\beta$.
Our goal is to bound $\inf_{\text{unitary }W}\magn{U-Q_1W}$.
Let $G_{ij}$ be the Gaussian random variable coupled with $\wt G_{ij}$ such that $|\wt G_{ij}-G_{ij}|\le\abs{G_{ij}}c_{\alg N}\meps$. Then
\(\rmagn{\wt G-G}\le\magn G\sqrt nc_{\alg N}\meps.\)
Let $\wt G=G+E_G$. Then for some $\magn F\le\rmagn{\wt A}\rmagn{\wt G}\mmu\meps$,
\begin{align*}
 M
  &=\alg{MM}(\wt A,\wt G)
\\&=\wt A\cdot\wt G+F
\\&=(U\Sigma V^*+E)\cdot(G+E_G)+F
\\&=U\Sigma V^*G+U\Sigma V^*E_G+E\cdot(G+E_G)+F.
\end{align*}
Set $X=U\Sigma V^*E_G+E\cdot(G+E_G)+F$.
Definition \ref{a11} guarantees that for some $\magn{E_M}\le\magn M\qrmu\meps$ and $\magn{E_Q}\le\qrmu\meps$ that
\begin{align*}
&(M+E_M)=(Q+E_Q)R
\\\implies &M=QR+E_QR-E_M
\\\implies &U\Sigma V^*G=QR+E_QR-E_M-X
\\\implies &U\Sigma V^*G_1=Q\bmat{R_{11}\\0}+E_Q\bmat{R_{11}\\0}-(E_M)_1-X_1
\\\implies &U\Sigma V^*G_1=Q_1R_{11}+(E_Q)_1R_{11}-(E_M)_1-X_1
\\\implies &U=Q_1R_{11}(V^*G_1)^{-1}\Sigma^{-1}+\sqbrac{(E_Q)_1R_{11}-(E_M)_1-X_1}(V^*G_1)^{-1}\Sigma^{-1}.
\end{align*}
We bound the first factor of the second term by
\begin{align*}
\magn{(E_Q)_1R_{11}}
  &\le\magn{E_Q}\magn{R}
\\&\le\qrmu\meps\cdot\magn{M+E_M}
\\&\le\qrmu\meps\cdot(1+\qrmu\meps)\magn M
\\&\le\qrmu\meps\cdot(1+\qrmu\meps)(1+\mmu\meps)\rmagn{\wt A}\rmagn{\wt G}
\\&\le\qrmu\meps\cdot(1+\qrmu\meps)(1+\mmu\meps)(\rmagn{\Sigma}+\beta)
(1+\nmu\meps)\magn G
\\&\le2\magn G(\rmagn{\Sigma}+\beta)\qrmu\meps
\\
\magn{(E_M)_1}
  &\le\magn{E_M}
\\&\le\qrmu\meps\cdot(1+\mmu\meps)\rmagn{\wt A}\rmagn{\wt G}
\\&\le2\magn G(\rmagn{\Sigma}+\beta)\qrmu\meps
\\
\magn{X_1}
  &\le\magn{X}
\\&\le\magn\Sigma\nmu\meps\cdot\magn G+\beta\cdot\rmagn{\wt G}+\rmagn{\wt A}\rmagn{\wt G}\mmu\meps
\\&\le\magn\Sigma\nmu\meps\cdot\magn G+\beta\cdot(1+\nmu\meps)\rmagn{G}+(\magn\Sigma+\beta)(1+\nmu\meps)\rmagn{G}\mmu\meps
\\&=  \magn G(\magn\Sigma+\beta)\sqbrac{(\nmu+\mmu)+\nmu\mmu\meps}\meps+\beta\rmagn{G}
\\&\le2(\nmu+\mmu)\magn G(\magn\Sigma+\beta)\meps+\beta\rmagn{G}
\end{align*}
Altogether, for $C=R_{11}(V^*G_1)^{-1}\Sigma^{-1}$ this gives
\spliteq{\label{a34}}{
\magn{U-Q_1C}
  &\le\magn{(E_Q)_1R_{11}-(E_M)_1-X_1}\cdot\magn{(V^*G_1)^{-1}\Sigma^{-1}}
\\&\le\sqbrac{\pare{4\qrmu+2\nmu+2\mmu}\pare{\magn\Sigma+\beta}\meps+\beta}\magn G\cdot\magn{(V^*G_1)^{-1}\Sigma^{-1}}
\\&\le\sqbrac{\pare{4\qrmu+2\nmu+2\mmu}\pare{\frac{\magn A+\beta}{\sigma_r(A)}}\meps+\frac\beta{\sigma_r(A)}}\cdot\magn G\magn{(V^*G_1)^{-1}}
\\&\le\sqbrac{\frac{\magn A+\beta+1}{\sigma_r(A)}}\cdot\magn G\magn{(V^*G_1)^{-1}}\cdot\beta
\\&=:m\cdot\beta
}
If $C$ was unitary we'd be nearly done. Unfortunately we have no such guarantee.
Instead we derive three different inequalities from (\ref{a34}).
\begin{align*}
(\ref{a34})
  \implies&\magn{U-Q_1C}\le m\beta
\\\implies&\big|\magn{Ux}-\magn{Q_1Cx}\big|\le m\beta\magn x
\\\implies&(1-m\beta)\magn x\le\magn{Q_1Cx}\le(1+m\beta)\magn x
\\\implies&\frac{1-m\beta}{1+\qrmu\meps}\le\frac{\magn{Cx}}{\magn x}\le\frac{1+m\beta}{1-\qrmu\meps}.
\end{align*}
In in particular, that implies bounds on the singular values of $C$, which we now apply.
\begin{align*}
(\ref{a34})
  \implies&\magn{U^*U-U^*Q_1C}\le m\beta
\\\implies&\big|\magn{x}-\magn{U^*Q_1Cx}\big|\le m\beta\magn x
\\\implies&(1-m\beta)\magn x\le\magn{U^*Q_1Cx}\le(1+m\beta)\magn x
\\\implies&\frac{1-m\beta}{\sigma_1(C)}\magn x\le\magn{U^*Q_1x}\le\frac{1+m\beta}{\sigma_r(C)}\magn x
\\\implies&\pare{1-\qrmu\meps}\frac{1-m\beta}{1+m\beta}\magn x\le\magn{U^*Q_1x}\le\frac{1+m\beta}{1-m\beta}\pare{1+\qrmu\meps}\magn x.
\\
(\ref{a34})
  \implies&\magn{(U^\perp)^*U-(U^\perp)^*Q_1C}\le m\beta
\\\implies&\magn{(U^\perp)^*Q_1Cx}\le m\beta\magn x
\\\implies&\magn{(U^\perp)^*Q_1x}\le\frac{1}{\sigma_r(C)}m\beta\magn x
\\\implies&\magn{(U^\perp)^*Q_1x}\le\frac{1+\qrmu\meps}{1-m\beta}m\beta\magn x
\end{align*}
Now consider the quantity we need to bound and apply the two above inequalities
\begin{align*}
    \inf_{\text{unitary }W}\magn{U-Q_1W}
      &\le\inf_{\text{unitary }W}\pare{\magn{I-U^*Q_1W}+\magn{(U^\perp)^*Q_1W}}
    \\&=\inf_{\text{unitary }W}\magn{W^*-U^*Q_1}+\magn{(U^\perp)^*Q_1}
    \\&=\max(\sigma_1(U^*Q_1)-1,1-\sigma_r(U^*Q_1))+\magn{(U^\perp)^*Q_1}
    \\&=\max\pare{\frac{1+m\beta}{1-m\beta}\pare{1+\qrmu\meps}-1,1-\pare{1-\qrmu\meps}\frac{1-m\beta}{1+m\beta}}+\frac{1+\qrmu\meps}{1-m\beta}m\beta
    \\&\le4m\beta+2m\beta
    \\&=  6m\beta.
\end{align*}
For the last inequality, we assumed that $m\beta\le0.2$.
We end with a tail estimate on $m$.
Note by rotational invariance that $V^*G_1$ is an $r\times r$ matrix of complex Gaussian entries. So applying Lemmas \ref{a32} and \ref{a33} we have by union bound that
\[\Pr\pare{m\ge\frac{\sigma_1(A)+2}{\sigma_r(A)}\cdot\frac{(2\sqrt2+t)\sqrt n}{x}}\le e^{-nt^2}+\frac r2x^2.\]
Note the complementary event implies $m\beta\le0.2$ by the requirement of $\beta$, so the desired bound holds.
\end{proof}

\begin{remark}
One may notice all we really need for spectral bisection is a basis of an invariant subspace of $A$ and wonder why we're bothering with matrix sign and deflate.
After all, we have ready-made methods for computing bases of invariant subspaces: Lanczos and subspace iteration.
For instance, consider \[X_0=\text{random}\in\C^{n\times (n/2)}\quad\quad\quad[X_{k+1},R_k]=\alg{QR}(AX_k).\]
Then $X_k$ converges to an invariant subspace of dimension $n/2$.
The issue is the rate of convergence. If $\lambda_1>\lambda_2>\cdots>\lambda_n\ge 0$ are the eigenvalues of $A$, then we only converge after
\[
\frac{\log(1/\eps)}{
\log\pare{\lambda_{n/2+1}/\lambda_{n/2}}
}
\approx
\frac{\log(1/\eps)}{
1-{\lambda_{n/2+1}/\lambda_{n/2}}
}
\] iterations. This is a \textit{polynomial} dependence on the relative eigenvalue gap, which is totally unacceptable. Lanczos offers only a square root improvement over this.
Since spectral projectors are square with a functional calculus, we have `repeated squaring' flavor algorithms which 
simulate computing something akin to $X_{2^k}$ or even $X_{3^k}$ in just $k$ iterations.
\end{remark}

\section{Spectral bisection}
\label{a7}
We're now ready to describe our spectral bisection algorithm, Algorithm \ref{a35} \thealg. As mentioned in the introduction, this work makes three main changes to the version of spectral bisection appearing in \cite{b0}. One is to use Newton-Schulz for sign estimation. The other two are highlighted in Remarks \ref{a8} and \ref{a9}.
\newcommand{\signp}{\delta_{\alg{SIGN}}}
\newcommand{\gbeta}{\frac{\rho\cdot\eps^2}{\ell^2\cdot n^3}}
\begin{algorithm}
\caption{$[U,D]=\thealg(A,R_0,R,\eps,\ell,\rho)$}
\label{a35}
\begin{algorithmic}[1]
    \Require Hermitian matrix $A$ with floating point entries. Initial size $R_0\ge\magn A$. Window size $R\ge\magn A$. Target accuracy $\eps>0$.
    Integer $\ell\ge20$.
    Failure parameter $\rho>0$.
    \Ensure
    With some probability, all singular values of $U$ lie in $[1-\eps/3,1+\eps/3]$ and $\magn{UDU^*-A}\le\eps\cdot(R_0+R)$.
    \If{$n=1$}
        \State\Return $(\bmat1,\,a_{11})$
    \EndIf
    \If{$R\le\eps R_0$}
        \State\Return $(I_{n\times n},\,0)$
    \EndIf
    \State$\eps'\gets(1-1/\ell)\eps$
        \label{a36}
    \State$R'\gets(1/2+2/\ell)R$
        \label{a37}
    \State
    \[\delta\gets\frac34\cdot\frac{\rho^{1/2}\eta}{n}\cdot\frac13\cdot\frac1{12\sqrt2+6\sqrt{\log(4/\rho)/n}}\text{ where }\eta=\frac{\eps'}{5\ell}\]
        \label{a38}
    \State$c=\unif(R/\ell)$
        \label{a39}
    \State$A_{\alg{shift}}\gets A-cI$
        \label{a40}
    \State$B=\alg{SIGN}(A_{\alg{shift}},\delta, 2R)$
        \label{a41}
    \State$P_\pm\gets(I\pm B)/2$
        \label{a42}
    \State$k_\pm=\alg{round}(\tr P_\pm)$
        \label{a43}
    \If{$k_\pm=n$}
    \Comment{Note $k_++k_-=n$ so this condition is met for at most one of $k_+$ and $k_-$.}
    \State $A_\pm\gets A\mp(R/2) I$
    \State$(U,D)=\thealg(A_\pm,R_0,R',\eps',\ell+1,\rho)$
    \State\Return$(U,D\pm(R/2)I)$
    \EndIf
    \State$Q_\pm=\alg{DEFLATE}(P_\pm,\,k_\pm)$
        \label{a44}
    \State$C_\pm=\alg{MM}(\alg{MM}(Q_\pm^*, A),Q_\pm)$
        \label{a45}
    \State$A_\pm\gets C_\pm\mp(R/2)I$
        \label{a46}
    \State$(U_\pm,D_\pm)=\thealg(A_\pm,R_0,R',\eps',\ell+1,\rho)$
        \label{a47}
    \State$W_\pm=\alg{MM}(Q_\pm,U_\pm)$
        \label{a48}
    \State
    \Return$\pare{\bmat{W_+&W_-},\bmat{D_++(R/2)I\\&D_--(R/2)I}}$\label{a49}
\end{algorithmic}
\end{algorithm}
In the pseudocode, the symbol $\leftarrow$ is used to denote floating-point assignment. That is, $x\gets r$ means $x=\alg{fl}(r)$ is the floating point number closest to $r$.
The ``root call'' to \thealg\,is
\[
[U,D]=\alg{EIGH}(A,\eps,\theta)
=
\thealg(A,\magn A,\magn A,\eps,\ceil{\lg(1/\eps)}+5,\theta/4n).
\]
\begin{remark}(Recursive parameters)
\label{a8}
\cite{b0} passes $0.8\cdot\eps$ in for $\eps$ in the recursive calls. For a computation tree of $O(\log(n))$ depth, this makes the $\eps$ seen by a deep node $\eps/\poly(n)$. This contributes to the precision $\meps$ required by that node, which itself is is a polynomial in $\eps$.
Algorithm \ref{a35} passes in $(1-1/\ell)\eps$ and increments $\ell$.  Instead of $0.8^k$, one ends up with a telescoping product resulting in only a constant factor smaller $\eps$ seen at the leaves.
\end{remark}
\begin{remark}(Shattering/Split points/Base case)
\label{a9}
\cite{b0} uses binary search to find a good value of the split point $c'$ that ensures $k_\pm$ are both at least $n/5$. This ensures that the sub-problems are constant factors smaller than the original, guaranteeing that recursion halts at depth $\log_{5/4}(n)$.
\cite{b19} uses the median of the diagonal entries to find a split point that ensures $k_\pm\ge1$. In the worst case, this leads to a depth of $n$.
We pick our split points $c'$ randomly, as suggested by \cite{b20}.
In order to avoid having to perturb the input matrix, the avoidance of binary search and adoption of random split points are both necessary.
If we picked $c'$ deterministically, an adversarial input may place an eigenvalue exactly at our selection preventing $\alg{SIGN}$ from converging.
We also cannot assume that our eigenvalues are spaced out. In particular, there may not even exist a split point with $k_\pm\ge n/5$ so binary searching for one would fail. The random split points are picked close to the midpoint, so they reduce the range of the spectrum by almost a factor of two each time.
\end{remark}

There are four quantities we need to bound for $\alg{EIGH}$. 1. Residual error 2. Success probability 3. Precision requirement 4. Runtime.
We do so in two parts. First, we analyze those for quantities within each call to $\thealg$ treating the recursive calls as oracle queries; we call this the ``local'' guarantee.
In particular, applying the local guarantee to the root call ensures the residual error of $\alg{EIG}$ satisfies desired bound.
For the other three quantities, we need a global analysis. In particular, we apply union bound to obtain the probability every call to $\thealg$ succeeds, sum the relevant runtimes, and take the minimum over all the precisions required.

\begin{lemma}(Local guarantee)
\label{a50}
Consider a call to \thealg\, with inputs $(A,R,\eps,\ell,\rho)$.
Fix any $w>0$ and set
\[
t=12\sqrt2+6\sqrt{\log(4/\rho)/n}
\quad\And\quad
\beta=\frac{\rho^{1/2}\eta}{3nt}.\]
Let $c'$ be the sample from the real interval $[-R/\ell,R/\ell]$ coupled with $c=\alg{UNIF}(R/\ell)$ in Step \ref{a39} so that $\magn{c-c'}\le(R/\ell)\meps$.
When using precision,
\[
\meps\le\min\pare{
\udeflate(\beta,n),
\usign(\beta/2,2R,2/w),
\frac14\cdot\frac\beta n\cdot\frac wR
}
\]
if one conditions on
\begin{enumerate}
    \item[]Event I: $c'\not\in\Lambda_w(A)$,
    \item[]Event II: Step \ref{a44} succeeds with error at most $\eta$, and
    \item[]Event III: The recursive calls in Step \ref{a47} succeed (i.e. produce $U_\pm,D_\pm$ with small residual),
\end{enumerate}
then the values returned in Step \ref{a49} have the advertised guarantee. Furthermore Event I occurs with probability at least $1-wn/R$ and Event II occurs with probability at least $1-2\rho$.
\end{lemma}
\begin{proof}
Our first task is to quantify $\magn{A_{\alg{shift}}-(A-c'I)}$.
\begin{equation}\label{a51}
\begin{split}
    \magn{A_{\alg{shift}}-(A-c'I)}
      &\le\magn{A_{\alg{shift}}-(A-cI)}+\magn{(A-cI)-(A-c'I)}
    \\&\le\magn{(A-cI)}\meps+\abs{c-c'}
    \\&\le(R+R/\ell)\meps+(R/\ell)\meps
    \\&\le(1+2/\ell)R\meps.
\end{split}
\end{equation}
Our next task is to quantify $\magn{P_\pm-\frac{\sign(A-c'I)\pm I}2}$
\begin{align}
    \magn{P_\pm-\frac{\sign(A-c'I)\pm I}2}
    &=
    \label{a52}
    \magn{P_\pm-\frac{B\pm I}2}
    \\&+
    \label{a53}
    \magn{\frac{B\pm I}2-\frac{\sign(A_{\alg{shift}})\pm I}2}
    \\&+
    \label{a54}
    \magn{\frac{\sign(A_{\alg{shift}})\pm I}2-\frac{\sign(A-c'I)\pm I}2}
\end{align}
We bound each of the three terms starting with (\ref{a54}). Note we may pull out $\frac12$ and the identity terms cancel.
Then Event I implies that we may apply Lemma \ref{a24} to bound it by
\[
(\ref{a54})=\frac12\magn{{\sign(A_{\alg{shift}})}-{\sign(A-c'I)}}\le
\frac12\cdot n\cdot\frac{(1+2/\ell)R\meps}{w-(1+2/\ell)R\meps}.\]
Note incidentally that $\Lambda_w(A)$ is a set of measure at most $2nw$ and $c'$ is sampled from a density bounded by $\ell/2R$ so the probability Event I fails is at most $2nw\cdot\ell/2R=nw/R$.
Next we consider (\ref{a53}). Note again we may pull out $\frac12$ and the identity terms cancel, so
\[
(\ref{a53})=\frac12\magn{\alg{SIGN}(A_{\alg{shift}},\delta,2R)-\sign(A_{\alg{shift}})}.
\]
In order to apply Theorem \ref{a29}, we must verify $\meps$ is small enough given the inputs.
First, (\ref{a51}) implies
\[\magn{A_{\alg{shift}}}
\le\magn{A_{\alg{shift}}-(A-c'I)}+\magn{A-c'I}
\le\sqbrac{(1+2/\ell)R\meps}+\sqbrac{(1+1/\ell)R}\le 2R.
\]
Second, Event I means $c'\not\in\Lambda_w(A)$ so
\[
\magn{A_{\alg{shift}}^{-1}}
\le
\frac1{\frac1{\magn{(A-c'I)^{-1}}}-\magn{A_{\alg{shift}} - (A-c'I)}}
\le\frac1{w-(1+2/\ell)R\meps}
\le\frac2{w}.
\]
Our selection of $\meps$ is bounded by \(\meps_{\alg{SIGN}}(\delta,2R,2/w)\), so Theorem \ref{a29} bounds $(\ref{a53})\le\delta/2$.
Finally
\[
(\ref{a52})
\le\frac{\magn B+1}2\meps
\le\frac{(1+\delta)+1}2\meps
\le(1+\delta/2)\meps.\]
Summing (\ref{a52})$+$(\ref{a53})$+$(\ref{a54}) gives
\[\magn{P_\pm-\frac{\sign(A-c'I)\pm I}2}
\le\sqbrac{\frac n2\frac{(1+2/\ell)\frac Rw}{1-(1+2/\ell)\frac Rw\meps}+1+\frac\delta2}\meps+\frac\delta2
\le\delta
\]
since $\meps\le\frac14\cdot\frac\beta n\cdot\frac wR$.
Note that $\frac{\sign(A-c'I)\pm I}2$ are the true spectral projectors into the the parts of the spectrum of $A$ above and below $c'$.
Since $\beta+n\meps\ll1/2$, the estimates $k_\pm$ are the exact ranks of $\frac{\sign(A-cI)\pm I}2$. Additionally, we have $\meps\le\udeflate(n,\beta)$.
These are exactly the conditions for $\alg{DEFLATE}$ to succeed with probability $1-\rho$. Applying union bound gives the probability both calls to $\alg{DEFLATE}$ succeed. Since we are conditioning on Event II, we have that
\begin{equation}
\label{a55}
\magn{Q_\pm-V_\pm}\le\eta\And\magn{Q_\pm}\le1+\qrmu\meps
\end{equation}
for some choice of $V_\pm$ having orthonormal columns with \(\sign\pare{A-c'I}=V_+V_+^*-V_-V_-^*.\)
We now show the recursive calls to \thealg\, in Step \ref{a47} satisfy the input requirements\----namely, we must ensure $R'\ge\magn{A_\pm}$.
\begin{align}
    \magn{A_\pm}
      &\le\magn{V_\pm^*\pare{A\mp\frac R2I}V_\pm}\label{a56}
    \\&+  \magn{A_\pm-V_\pm^*\pare{A\mp\frac R2I}V_\pm}\label{a57}.
\end{align}
Since $V_\pm$ are the bases for the upper and lower parts of the spectrum of $A$, the spectrums of $V_\pm AV_\pm$ are contained in $[c',R]$ and $[-R,c']$ respectively, and so the spectrums of $V_\pm\pare{A\mp\frac R2I}V_\pm$ are contained in $[c-R/2,R/2]$ and $[-R/2,c+R/2]$ respectively, both of which are contained in $[-(1/2+1/\ell)R, (1/2+1/\ell)R]$ since $\abs c\le R/\ell$. This results in (\ref{a56})$\le(1/2+1/\ell)R$. Bounding (\ref{a57}) is more involved.
\begin{align}
\magn{A_\pm-V_\pm^*\pare{A\mp\frac R2I}V_\pm}
  &=  \magn{A_\pm-\pare{C_\pm\mp\frac R2I}}+\magn{\pare{C_\pm\mp\frac R2I}-V_\pm^*\pare{A\mp\frac R2I}V_\pm}
\\&\le\pare{\magn{C_\pm}+\frac R2}\meps+\magn{C_\pm-V_\pm^*AV_\pm}
\\&\le\pare{\magn{C_\pm}+\frac R2}\meps
+\magn{C_\pm-Q_\pm^*AQ_\pm}
+\magn{Q_\pm^*AQ_\pm-V_\pm^*AV_\pm}\label{a58}.
\end{align}
The first term of (\ref{a58}) can be bounded using
\begin{equation}
\label{a59}    
\magn{C_\pm}\le\magn{C_\pm-Q_\pm^*AQ_\pm}+\magn{Q_\pm^*AQ_\pm-V_\pm^*AV_\pm}+\magn{V_\pm^*A V_\pm}.
\end{equation}
Then both the first term of (\ref{a59}) and second term of (\ref{a58}) can be bounded by using the guarantee for $\alg{MM}$ twice,
\begin{equation}
\begin{split}
\magn{C_\pm-Q_\pm^*AQ_\pm}
  &\le\mmu\magn A\magn{Q_\pm}^2(2+\mmu\meps)\cdot\meps
\\&\le\mmu\cdot R\cdot\pare{1+\qrmu\meps}^2\cdot(2+\mmu\meps)\cdot\meps
\\&\le3R\mmu\meps.
\end{split}
\end{equation}
Then both the second term of (\ref{a59}) and third term of (\ref{a58}) can be bounded using (\ref{a55})
\begin{equation}
\begin{split}
\magn{Q_\pm^*AQ_\pm-V_\pm AV_\pm}
  &=  \magn{(V_\pm+\eta E'_\pm)^*A(V_\pm+\eta E'_\pm)-V_\pm AV_\pm}
\\&\le\magn{V_\pm^*AE'_\pm\eta+E'_\pm AV_\pm\eta+E'_\pm AE'_\pm\eta^2}
\\&\le  (2+\eta)R\eta.
\end{split}
\end{equation}
Summing these estimates and rearranging gives
\spliteq{\label{a60}}{
(\ref{a57})\le
(\ref{a58})
  &\le\sqbrac{(1+\meps)\cdot3\mmu+(3/2)}R\meps+(1+\meps)\cdot(2+\eta)R\eta
\\&\le(2+2\eta)R\eta
}
and consequently 
\[\magn{A_\pm}\le(\ref{a56})+(\ref{a57})\le\pare{\frac12+\frac1\ell+(2+\eta)\eta}R\le\pare{\frac12+\frac2\ell}(1-\meps)R\le R'.\]
Let's now show that $\bmat{W_+&W_-}$ is nearly unitary.
By (\ref{a55}), we may express $Q_\pm=V_\pm+\eta E_\pm'$ for some $\magn{E_\pm'}\le1$.
By Definition \ref{a12}, we may express
\begin{equation}\label{a61}
\begin{split}
    W_\pm
  &=Q_\pm U_\pm+\magn{Q_\pm}\magn{U_\pm}\mmu \meps E
\\&=V_\pm U_\pm+\eta E_\pm'U_\pm+\magn{Q_\pm}\magn{U_\pm}\mmu \meps E
\end{split}
\end{equation}
for some $\magn E\le1$.
Set $X_\pm=W_\pm-V_\pm U_\pm$ and note we have the estimate
\begin{equation}\label{a62}
\begin{split}
    \magn{X_\pm}
      &\le\eta\magn{U_\pm}+\magn{Q_\pm}\magn{U_\pm}\mmu\meps
    \\&\le\eta(1+\eps')+(1+\qrmu\meps)(1+\eps')\mmu\meps
    \\&\le\eta(1+2\eps')
\end{split}
\end{equation}
By construction, we have
\[\bmat{W_+&W_-}=\bmat{V_+&V_-}\bmat{U_+\\&U_-}+\bmat{X_+&X_-}.\]
Note that $\bmat{V_+&V_-}$ is a unitary matrix and that Event III guarantees the singular values of $U_\pm$ lie in $[1-\eps'/3,1+\eps'/3]$. So the singular values of $\bmat{W_+&W_-}$ satisfy
\[
(1-\eps'/3)-\sqrt2\max_\pm\magn{X_\pm}
\le
\sigma_j\pare{\bmat{W_+&W_-}}
\le
(1+\eps'/3)+\sqrt2\max_\pm\magn{X_\pm}.
\]
Finally, applying (\ref{a62}) gives
\begin{align*}
      \frac{\eps'}3+\sqrt2\max_\pm\magn{X_\pm}
  &=  \frac{\eps'}3+\sqrt2\eta(1+2\eps')
\\&=  \pare{1+{3\sqrt2}\cdot\frac{\eta}{\eps'}\cdot(1+2\eps')}\frac{\eps'}3
\\&\le\pare{1+\frac1\ell}\frac{\eps'}3
\\&\le\pare{1+\frac1\ell}(1+\meps)\pare{1-\frac1\ell}\frac{\eps}3
\\&\le\eps/3
\end{align*}
so each singular value of $\bmat{W_+&W_-}$ is in the interval $[1-\eps/3,1+\eps/3]$ as required. Our final task is showing our output has sufficiently small residual.
Let $G_\pm=D_\pm\pm\frac R2I$ and $\wt G_\pm=\alg{fl}(G_\pm)$.
The approximate factorization of $A$ returned in Step \ref{a49} is
\[
\bmat{W_+&W_-}\bmat{\wt G_+\\&\wt G_-}\bmat{W_+&W_-}^*
=W_+\wt G_+W_+^*
+W_-\wt G_-W_-^*.
\]
The analogous exact factorization is \[
A
=\bmat{V_+&V_-}\bmat{V_+^*AV_+\\&V_-^*AV_-}\bmat{V_+&V_-}^*
=V_+V_+^*AV_+V_+^*+V_-V_-^*AV_-V_-^*.
\]
We define two additional factorizations, namely the sums $\wt F_++\wt F_-$ and $F_++F_-$ for
\[
\wt F_\pm:=(V_\pm U_\pm)\wt G_\pm (V_\pm U_\pm)^*
\quad\And\quad
F_\pm:=(V_\pm U_\pm)G_\pm (V_\pm U_\pm)^*
\]
Then the final residual is bounded by
\begin{align}
    \magn{A-\bmat{W_+&W_-}\bmat{\wt G_+\\&\wt G_-}\bmat{W_+&W_-}^*}
  &\le
    \magn{A-(F_++F_-)}\label{a63}
\\&+
    \magn{(F_++F_-)-(\wt F_++\wt F_-)}\label{a64}
\\&+
    \magn{\wt F_+-W_+\wt G_+W_+^*}+\magn{\wt F_--W_-\wt G_-W_-^*}\label{a65}
\end{align}
We begin by bounding each term in (\ref{a65}).
Event III guarantees \(\magn{A_\pm-U_\pm D_\pm U_\pm^*}\le\eps'\) which implies $\magn{D_\pm}\le\frac{\magn{A_\pm}+\eps'}{(1-\eps')^2}\le\frac{R'+\eps'}{(1-\eps')^2}$.
Then 
\begin{equation}
\begin{split}
\magn{\wt F_\pm-W_\pm\wt G_\pm W_\pm^*}
  &=\magn{(V_\pm U_\pm)\wt G_\pm(V_\pm U_\pm)^*-W_\pm \wt G_\pm W_\pm^*}
\\&=\magn{(V_\pm U_\pm)\wt G_\pm(V_\pm U_\pm)^*-(V_\pm U_\pm+X_\pm)\wt G_\pm(V_\pm U_\pm+X_\pm)^*}
\\&=\magn{-X_\pm\wt G_\pm(V_\pm U_\pm)^*-(V_\pm U_\pm)\wt G_\pm X_\pm^*-X_\pm\wt G_\pm X_\pm^*}
\\&\le(2+2\eps'+\magn{X_\pm})\rmagn{\wt G_\pm}\magn{X_\pm}
\\&\le(2+2\eps'+(1+2\eps')\eta)\pare{\frac{R'+\eps'}{(1-\eps')^2}+\frac R2}(1+\meps)\cdot(1+2\eps')\eta
\end{split}
\end{equation}
Next we bound (\ref{a64}).
\begin{align*}
    \magn{(F_++F_-)-(\wt F_++\wt F_-)}
  &=\magn{\bmat{V_+&V_-}^*\pare{F_++F_--\wt F_+-\wt F_-}\bmat{V_+&V_-}}
\\&=\magn{\bmat{U_+(G_+-\wt G_+)U_+^*\\& U_-(G_--\wt G_-)U_-^*}}
\\&\le\max_\pm\magn{U_\pm}^2\magn{G_\pm}\meps
\\&\le(1+\eps'/3)^2\pare{\frac{R'+\eps'}{(1-\eps')^2}+\frac R2}\meps
\\&\le2R\meps
\end{align*}
Finally, we bound (\ref{a63}).  Since $\bmat{V_+&V_-}$ is unitary, conjugation by it does not change norms.
\begin{align*}
      \magn{A-(F_++F_-)}
  &=  \magn{\bmat{V_+&V_-}^*\pare{A-(F_++F_-)}\bmat{V_+&V_-}}
\\&=  \magn{\bmat{V_+^*AV_+\\&V_-^*AV_-}-\bmat{U_+ G_+ U_+^*\\&U_- G_- U_-^*}}
\\&=  \max_\pm\magn{V_\pm^*AV_\pm-U_\pm G_\pm U_\pm^*}
\\&=  \max_\pm\magn{V_\pm^*AV_\pm-U_\pm \pare{D_\pm\pm\frac R2I} U_\pm^*}
\\&=  \max_\pm\magn{V_\pm^*AV_\pm-U_\pm D_\pm U_\pm^*\mp\frac R2U_\pm U_\pm^*}
\end{align*}
We bound that final expression as the sum of three terms
\begin{align}
    \magn{V_\pm^*AV_\pm-U_\pm D_\pm U_\pm^*\mp\frac R2U_\pm U_\pm^*}
      &\le
    \magn{V_\pm^*AV_\pm\mp\frac R2I-A_\pm}\label{a66}
    \\&+
    \magn{A_\pm-U_\pm D_\pm U_\pm^*}\label{a67}
    \\&+
    \magn{\frac R2U_\pm U_\pm^*-\frac R2 I}\label{a68}
\end{align}
We obtained a bound for (\ref{a66}) in (\ref{a60}), which is $(2+2\eta)R\eta$.
Event III is exactly that (\ref{a67}) is bounded by $(R'+R_0)\eps'$ and that (\ref{a68}) is bounded by $\frac R2[(1+\eps'/3)^2-1]\le0.34\eps'R$. The final residual then is bounded by the sum
\begin{align*}
 (\ref{a66})
+(\ref{a67})
+(\ref{a68})
+(\ref{a64})
+(\ref{a65})
  &\le
(4+10\eps)R\eta
+(\ref{a67})
+(\ref{a68})
\\&\le
(4+10\eps)R\eta
+(R'+R_0)\eps'
+0.34R\eps'
\\&\le
\pare{\frac{4+10\eps}{5\ell}
+(1+\meps)\pare{\frac12+\frac2\ell}
+0.34}R\eps'+R_0\eps'
\\&\le(R+R_0)\eps'
\\&\le(R+R_0)\eps
\end{align*}
as required.
\end{proof}

\begin{theorem}[Main guarantee]
\label{a4}
Let $A$ be an $n\times n$ Hermitian matrix, $\theta\in(16ne^{-7.4n},1)$, and $\eps\in(0,2^{-15})$.
Using
\begin{align*}
\lg(1/\meps)
  \ge\lg(1/\eps)&+\lg\max\pare{n^{1.5}\qrmu,n^{1.5}\nmu,n^{4.5}\mmu}
\\&+2\lg\lg(1/\eps)+1.5\lg(1/\theta)+\lg\lg(n\lg(1/\eps)/\theta)+23
\end{align*}
bits of precision and 
\[
O\big(
\log(1/\eps)\pare{\log(n)+\log(1/\theta)+\log\log(1/\eps)}T_{\alg{MM}}(n)
+
\log(1/\eps)T_{\alg{QR}}(n)
\big)
\]
\flops, the function call \[
[U,D]=\alg{EIGH}(A,\eps,\theta)
=
\thealg(A,\magn A,\magn A,\eps,\ceil{\lg(1/\eps)}+5,\theta/4n)
\]
achieves the following guarantee with probability at least $1-\theta$. First, $\magn{A-UDU^*}\le2\eps\magn A$ and second, all the singular values of $U$ lie in the interval $[1-\eps/3,1+\eps/3]$.
\end{theorem}
\begin{proof}
Consider the computation tree associated with the recursive structure of this algorithm. 
It is a binary tree where each node denotes a call to \thealg.
The number of children of each node is the number of recursive calls to \thealg\, it makes. Let $d$ be the depth of the tree. Let $X$ be the set of all nodes and $X_D$ the set of nodes in which deflate is run.
The values of $R_0=\magn A$ and $\rho=\theta/4n$ used throughout do not change, so we may use those variables without ambiguity. For the other parameters, let $R_x,\eps_x,\ell_x$ denote the values of $R,\eps,\ell$ input to node $x\in X$.
We omit the subscript when $x$ is the root node.
In particular, $R=\magn A$ and $\ell=\ceil{\lg(1/\eps)}+5$.
Set $w_x={\theta\cdot R_x}/\pare{2\cdot n_x\cdot nd}$.
Let $\beta_x,t_x,\eta_x$ be the values as specified by Lemma \ref{a50} and \thealg.
Note that for each node, Event III is implied by Events I-III for its children. Further note Event III is trivially satisfied for leaves. Consequently, it suffices to bound the probabilities of Events I and II for each node.
For $x\in X_D$, Lemma \ref{a50} implies this probability is
\[1-\frac{n_xw_x}{R_x}-2\rho_x=1-\frac{\theta}{2nd}-2\rho.\]
For $x\in X\backslash X_D$, only Event I is necessary, so the relevant probability is just
\[1-\frac{n_xw_x}{R_x}=1-\frac{\theta}{2nd}.\]
Note that $X_D$ are the internal nodes of the computational tree with two children. 
Since this tree has at most $n$ leaves, this means $\abs{X_D}=n-1$.
The number of nodes at each depth is at most $n$, so $\abs{X}\le nd$.
So by union bound the failure probability is at most
\begin{align*}
    \pare{\sum_{x\in X}\frac{\theta}{2nd}}+\pare{\sum_{x\in X_D}2\rho}
  &\le
    \abs{X}\cdot\frac{\theta}{2nd}+\abs{X_D}\cdot2\rho
\\&\le
    \frac\theta2+(n-1)\cdot2\cdot\frac{\theta}{4n}
\\&\le\theta.
\end{align*}
We now turn our attention to the number of bits used.
If $y$ is a child of $x$, then $\ell_y=\ell_x+1$ and
\begin{align}
R_y
&\le(1+\meps)\pare{\frac12+\frac2{\ell_x}}R_x=(1+\meps)\cdot\frac12\cdot\frac{\ell_x+4}{\ell_x}\cdot R_x,
\\
\eps_y
&\ge(1-\meps)\pare{1-\frac1{\ell_x}}\eps_x
   =(1-\meps)\cdot\frac{\ell_x-1}{\ell_y-1}\cdot\eps_x.
\end{align}
In both cases, we obtain a telescoping product for a node $y$ at depth $k$ giving
\begin{align}
R_y&\le(1+\meps)^k\cdot 2^{-k}
\cdot\frac{\ell+k}{\ell}
\cdot\frac{\ell+k+1}{\ell+1}
\cdot\frac{\ell+k+2}{\ell+2}
\cdot\frac{\ell+k+3}{\ell+3}\cdot R,
\\&\le(1+\meps)^k\cdot2^{-k}\cdot(1+k/\ell)^4\cdot R
\\
\eps_y&\ge(1-\meps)^k\cdot\frac{\ell-1}{\ell+k-1}\cdot\eps.
\end{align}
We claim that the depth of the tree $d$ is no more than $\ell$. To see this, note that at depth $k=\ell=\ceil{\lg(1/\eps)}+5$, we have $R_y\le\eps R=\eps R_0$ which guarantees the algorithm terminates. We therefore obtain $\eps_y\ge(1-\meps)^\ell\cdot\frac{\ell-1}{2\ell-1}\eps\ge\eps/3$ and $\ell_y\le2\ell$ for all $y\in X$.
We need to bound the constraints on $\meps$ from Lemma \ref{a50} at each node. We first compute bounds in terms of $\beta_x$.
\spliteq{\label{a69}}{
    \min_{x\in X}\udeflate(\beta_x,n)
  &=  \min_{x\in X}\frac1{4\qrmu(n_x)+2\nmu(n_x)+2\mmu(n_x)}\cdot\beta_x
\\&\ge\frac1{4\qrmu(n)+2\nmu(n)+2\mmu(n)}\cdot\min_{x\in X}\beta_x
}
\spliteq{\label{a70}}{
\min_{x\in X}\usign(\beta_x/2,2R_x,2/w_x)
  &=
\min_{x\in X}\frac{w_x}{4R_x}\cdot\frac1{4\max\pare{N_{\alg{SCALAR}}\pare{\frac{w_x}{4R_x},\frac{\eps_x}{8n_x}}\cdot n_x\mu_{\alg g}(n_x,1.1),n_x^2}}\cdot\frac{\beta_x}2
\\&=
\frac1{64}\min_{x\in X}\frac{\theta}{n_xnd}\cdot\frac1{\max\pare{N_{\alg{SCALAR}}\pare{\frac{\theta}{8n_xnd},\frac{\eps_x}{8n_x}}\cdot n_x\mu_{\alg g}(n_x,1.1),n_x^2}}\cdot\beta_x
\\&\ge
\frac{\theta}{64n^2d}\cdot\frac1{\max\pare{N_{\alg{SCALAR}}\pare{\frac{\theta}{8n^2d},\frac{\eps/3}{8n}}\cdot n\mu_{\alg g}(n,1.1),n^2}}\cdot\min_{x\in X}\beta_x
\\&\ge
\frac{\theta}{64n^3\mu_{\alg g}(n,1.1)}\cdot\frac1{\ell\cdot N_{\alg{SCALAR}}\pare{\frac{\theta}{8n^2\ell},\frac{\eps}{24n}}}\cdot\min_{x\in X}\beta_x
\\&\ge
\frac{\theta}{64n^3\mu_{\alg g}(n,1.1)}\cdot\frac1{\ell\cdot
\pare{2.5+2\lg(8n^2\ell/\theta)+\lg\lg(24n/\eps)}
}\cdot\min_{x\in X}\beta_x
\\&\ge
\frac{\theta}{64n^3\mu_{\alg g}(n,1.1)}\cdot\frac1{\ell\cdot
\lg\pare{363n^4\ell^2\theta^{-2}\lg(24n/\eps)}
}\cdot\min_{x\in X}\beta_x
}
We now bound $\min_{x\in X}\beta_x$.
First note \begin{equation}
\label{a72}
    \min_{x\in X}\eta_x=\min_{x\in X}\frac{\eps'_x}{5\ell_x}\ge\frac{\eps/3}{10\ell}.
\end{equation}
Then note that
\[n_x\cdot t_x
=n_x\cdot(12\sqrt2+6\sqrt{\log(4/\rho)/n_x})
=12\sqrt2n_x+6\sqrt{n_x\log(16n/\theta)}
\]
is monotonically increasing in $n_x$, so the maximum is obtained for $n_x=n$. Apply that observation with (\ref{a72}) and $\theta\ge16ne^{-7.4n}$ to obtain
\spliteq{\label{a73}}{
\min_{x\in X}\beta_x
   =  \min_{x\in X}\frac{\rho^{1/2}\eta_x}{3n_xt_x}
   &= \frac{\theta^{1/2}}{6n^{1/2}}\cdot\min_{x\in X}\frac{\eta_x}{n_xt_x}
   \\&\ge\frac{\theta^{1/2}}{6n^{1/2}}\cdot\frac1{12\sqrt2n+6\sqrt{n\log(16n/\theta)}}\min_{x\in X}\eta_x
   \\&\ge\frac{\theta^{1/2}}{6n^{1/2}}\cdot\frac1{12\sqrt2n+6\sqrt{n\log(16n/\theta)}}\cdot\frac{\eps}{30\ell}
   \\&\ge\frac1{6000}\frac{\theta^{1/2}\eps}{n^{3/2}\ell}.
}
Combining (\ref{a73}) with (\ref{a69}) gives
\[
\min_{x\in X}\udeflate(n_x,\eta_x,\rho_x)\ge
2\times10^{-5}\cdot\frac{\theta^{1/2}\eps}{\ell n^{3/2}\max(\mmu,\nmu,\qrmu)}.
\]
Combining (\ref{a73}) with (\ref{a70}) and $\meps_{\alg g}(n,1.1)\le2\mmu$ for $\mmu\ge10$ from Lemma \ref{a23} gives
\spliteq{}{
\min_{x\in X}\usign(\beta_x/2,2R_x,2/w_x)
  &\ge2.6\times10^{-6}\cdot
\frac{\theta^{3/2}\eps}{\ell^2n^{9/2}\mu_{\alg g}(n,1.1)}\cdot\frac1{\lg\pare{363n^4\ell^2\theta^{-2}\lg(24n/\eps)}
}
\\&\ge1.3\times10^{-6}\cdot
\frac{\theta^{3/2}\eps}{\ell^2n^{9/2}\mmu}\cdot\frac1{\lg\pare{363n^4\ell^2\theta^{-2}\lg(24n/\eps)}
}
}
By taking logs, one sees that it suffices to take $\lg(1/\meps)$ to be at least the larger of the following two quantities.
\spliteq{}{
\lg(1/\meps)&\ge\lg(1/\eps)+\lg\pare{n^{3/2}\max\pare{\nmu,\qrmu}}+\lg(\ell)+0.5\lg(1/\theta)+15.7,
\\
\lg(1/\meps)&\ge\lg(1/\eps)+\lg(n^{9/2}\mmu)+2\lg(\ell)+1.5\lg(1/\theta)
+\lg\lg\pare{{363n^4\ell^2\lg\pare{{24n}/\eps}}/{\theta^2}}+19.6.
}
Plugging in $\ell=\ceil{\lg(1/\eps)}+5$ and simplifying gives the final precision result.
We conclude with analysis of the runtime. The number of \flops\,used by $\alg{SIGN}$ in node $x$ is 
\spliteq{\label{a74}}{
O\pare{
\pare{\lg\pare{\frac{R_x}{w_x}}+\lg\lg\pare{\frac{n_x}{\eps_x}}}T_{\alg{MM}}(n_x)
}
  &=
O\pare{
\pare{\lg\pare{\frac{n_xnd}{\theta}}+\lg\lg\pare{\frac{n}{\eps}}}T_{\alg{MM}}(n_x)
}
\\&=
O\pare{
\pare{\lg\pare{\frac n\theta}+\lg\lg\pare{\frac1\eps}}T_{\alg{MM}}(n_x)
}
\\&=:
f_{\alg{SIGN}}(n_x)
}
The number of \flops\,used by $\alg{DEFLATE}$ in node $x$ is
\spliteq{\label{a75}}{
T_{\alg{QR}}(n_x)+T_{\alg{MM}}(n_x)+O(n^2)
=:f_{\alg{DEFLATE}}(n_x)
}
The number of \flops\,used by all other steps in the algorithm are dominated by these terms.
Let $X_k$ be the set of nodes at depth $k$. Then $\sum_{x\in X_k}n_x\le n$.
Note $f_{\alg{SIGN}}$ and $f_{\alg{DEFLATE}}$ are convex, so
\[
\sum_{x\in X_k}\pare{f_{\alg{SIGN}}(n_x)+f_{\alg{DEFLATE}}(n_x)}
\le
f_{\alg{SIGN}}(n)+f_{\alg{DEFLATE}}(n).
\]
Summing over $k=1,\cdots,d=\lg\ceil{1/\eps}+5$ gives the final result.
\end{proof}
We conclude with a final remark about removing the $\lg(1/\theta)$ terms from the bit requirement.
\begin{remark}[Boosting success probability]
One can estimate the residual in $O(n^2)$ time by computing
\[\magn{(UDU^*-A)x}\]
for randomly selected $x$. If the residual is too high, one can rerun \alg{EIG} with fresh randomness. This allows one to boost the probability of success. If the desired failure probability is $\theta'$, one can take $\theta=1/2$ and repeat the call to \alg{EIGH} $\lg(1/\theta')$ times. So at the expense of a longer runtime, one can remove the $\lg(1/\theta)$ terms from the precision bound.
\end{remark}

\newpage
\bibliographystyle{alpha}
\bibliography{outbib}

\end{document}